\documentclass[english]{article}
\usepackage[british]{babel}
\usepackage{amsmath,amssymb,amsthm}
\usepackage{geometry}
\usepackage{url}

\usepackage{breakurl}
\usepackage[breaklinks]{hyperref}
\usepackage{color}
\usepackage{mathtools}
\usepackage{mathrsfs}

\usepackage{tikz}
\usetikzlibrary{shapes.geometric, arrows}
\usetikzlibrary{decorations.pathmorphing}
\usetikzlibrary{cd}
\tikzstyle{var} = [rectangle, minimum width=1cm, minimum height=0.5cm, text centered, draw=black, fill=white!30]
\tikzstyle{output} = [rectangle, text centered, draw=black, fill=gray!20]
\tikzstyle{input} = [rectangle, text centered, draw=black, fill=green!20]
\tikzstyle{arrow} = [thick,->,>=stealth]

\theoremstyle{definition}
\newtheorem{theorem}{Theorem}[subsection]

\newtheorem{corollary}[theorem]{Corollary}
\newtheorem{definition}[theorem]{Definition}
\newtheorem{digression}[theorem]{Digression}
\newtheorem{example}[theorem]{Example}
\newtheorem{exercise}[theorem]{Exercise}
\newtheorem{lemma}[theorem]{Lemma}
\newtheorem{notation}[theorem]{Notation}

\newtheorem{proposition}[theorem]{Proposition}
\newtheorem{remark}[theorem]{Remark}
\newtheorem{slogan}[theorem]{Slogan}

\newcommand{\thistheoremname}{}
\newtheorem{genericthm}[theorem]{\thistheoremname}
\newenvironment{customenvironment}[1]
  {\renewcommand{\thistheoremname}{#1}%
   \begin{genericthm}}
  {\end{genericthm}}

\renewcommand{\setminus}{\smallsetminus}
\newcommand{\minus}{\smallsetminus}

\newcommand{\iso}{\simeq}

\newcommand{\lrangle}[1]{\langle #1 \rangle}

\newcommand{\Aut}{\text{Aut}}

\newcommand{\Bez}{\text{B\'ez}}

\newcommand{\Bl}{\text{Bl}}
\renewcommand{\char}{\text{char}}

\newcommand{\EKL}{\text{EKL}}

\newcommand{\Fix}{\text{Fix}}

\newcommand{\Fun}{\text{Fun}}

\newcommand{\GW}{\text{GW}}
\newcommand{\Gr}{\text{Gr}}
\newcommand{\grad}{\text{grad}}
\newcommand{\Hom}{\text{Hom}}
\providecommand{\ind}{\text{ind}}
\newcommand{\Jac}{\text{Jac}}
\newcommand{\Map}{\text{Map}}
\newcommand{\op}{\text{op}}

\newcommand{\Pre}{\text{Pre}}
\newcommand{\Proj}{\text{Proj}\hspace{0.1em}}
\newcommand{\PSL}{\text{PSL}}
\newcommand{\rank}{\text{rank}}

\newcommand{\s}{\text{s}}
\newcommand{\sgn}{\text{sgn}}
\newcommand{\Sh}{\text{Sh}}
\newcommand{\sig}{\text{sig}}
\newcommand{\Sm}{\textit{Sm}}
\newcommand{\SO}{\text{SO}}

\newcommand{\Spc}{\text{Spc}}
\newcommand{\Spec}{\text{Spec}\hspace{0.1em}}
\newcommand{\Spin}{\text{Spin}}
\providecommand{\spn}{\text{span}}
\newcommand{\sSet}{\textit{sSet}}

\newcommand{\Sym}{\text{Sym}}

\newcommand{\Th}{\text{Th}}
\newcommand{\Tr}{\text{Tr}}
\newcommand{\type}{\text{type}}

\newcommand{\A}{\mathbb{A}}
\newcommand{\C}{\mathbb{C}}
\newcommand{\F}{\mathbb{F}}
\newcommand{\Q}{\mathbb{Q}}
\newcommand{\R}{\mathbb{R}}
\newcommand{\Z}{\mathbb{Z}}

\providecommand{\CP}{\mathbb{C}\mathrm{P}}
\renewcommand{\O}{\mathcal{O}}
\renewcommand{\H}{\textbf{H}}

\let\emptyset\varnothing

\let\phi\varphi
\let\bar\overline
\let\tilde\widetilde
\newcommand{\gw}[1]{\left\langle #1 \right\rangle}

\newcommand{\G}{\mathbb{G}}
\renewcommand{\P}{\mathbb{P}}

\newcommand{\po}{\arrow[ul,phantom,"\ulcorner" very near start]}

\newcommand{\xto}[1]{\xrightarrow{#1}}
\newcommand{\from}{\leftarrow}
\newcommand{\xfrom}[1]{\overset{#1}{\leftarrow}}

\makeatletter
\newcommand{\superimpose}[2]{%
  {\ooalign{$#1\@firstoftwo#2$\cr\hfil$#1\@secondoftwo#2$\hfil\cr}}}
\makeatother
\newcommand{\smallslash}{\mbox{\tiny/}}

\newcommand{\clhook}{\mathrel{\raisebox{0.1em}{$\mathrel{\mathpalette\superimpose{{\hspace{0.1cm}\vspace{0.1em}\smallslash}{\hookrightarrow}}}$}}}

\newcommand{\ohook}{\mathrel{\raisebox{0.03em}{$\mathrel{\mathpalette\superimpose{{\hspace{0.1cm}\vspace{0.03em}\mbox{\small$\circ$}}{\hookrightarrow}}}$}}}

\usepackage{caption}
\usepackage{subcaption}
\usepackage{longtable}

\usepackage{mdframed}
\usepackage{float}

\usepackage[backend=biber,style=alphabetic]{biblatex}
\providecommand{\RP}{\mathbb{R}\textbf{P}}
\providecommand{\sPre}{\text{sPre}}
\providecommand{\smashprod}{\wedge}
\providecommand{\Top}{\text{Top}}
\renewcommand{\top}{\text{top}}
\providecommand{\hocolim}{\text{hocolim}}
\providecommand{\PGL}{\text{PGL}}
\providecommand{\Nis}{\text{Nis}}
\begin{document}

\title{An Introduction to $\A^1$-Enumerative Geometry \\ \large Based on lectures by Kirsten Wickelgren delivered at LMS-CMI ``Homotopy Theory and Arithmetic Geometry: Motivic and Diophantine Aspects''}
\author{Thomas Brazelton}
\date{2020}
\maketitle

\begin{abstract}
\noindent We provide an expository introduction to \textit{$\A^1$-enumerative geometry}, which uses the machinery of $\A^1$-homotopy theory to enrich classical enumerative geometry questions over a broader range of fields. Included is a discussion of enriched local degrees of morphisms of smooth schemes, following Morel, $\A^1$-Milnor numbers, as well as various computational tools and recent examples.
\end{abstract}

\tableofcontents{}

\section*{Introduction}
In the late 1990's Fabien Morel and Vladimir Voevodsky investigated the question of whether techniques from algebraic topology, particularly homotopy theory, could be applied to study varieties and schemes, using the affine line $\A^1$ rather than the interval $[0,1]$ as a parametrizing object. This idea was influenced by a number of preceding papers, including work of Karoubi and Villamayor \cite{Karoubi} and Weibel \cite{Weibel} on $K$-theory, and Jardine's work on algebraic homotopy theory \cite{Jardine1,Jardine2}. In work with Suslin developing an algebraic analog of singular cohomology which was $\A^1$-invariant \cite{Suslin-Voevodsky}, Voevodsky laid out what he considered to be the starting point of a homotopy theory of schemes parametrized by the affine line \cite{Voevodsky-ICM}.
This relied upon Quillen's theory of model categories \cite{Quillen}, which provided the abstract framework needed to develop homotopy theory in broader contexts.
Following this work, Morel \cite{Morel-htpy-schemes} and Voevodsky \cite{Voevodsky-ICM} constructed equivalent unstable $\A^1$-homotopy categories, laying the groundwork for their seminal paper \cite{MV99} which marked the genesis of \textit{$\A^1$-homotopy theory}.
Since its inception, this field of mathematics has seen far-reaching applications, perhaps most notably Voevodsky's resolution of the Bloch-Kato conjecture, a classical problem from number theory \cite{Voev-bloch-kato}.

The machinery of $\A^1$-homotopy theory works over an arbitrary field $k$ (in fact over arbitrary schemes, and even richer mathematical objects), allowing enrichments of classical problems which have only been explored over the real and complex numbers. Recent work in this area has generalized classical enumerative problems over wider ranges of fields, forming a body of work which we are referring to as \textit{$\A^1$-enumerative geometry}.

Beginning with a recollection of the topological degree for a morphism between manifolds in Section~\ref{sec:top-deg}, we pursue an idea of Barge and Lannes to produce a notion of degree valued in the Grothendieck--Witt ring of a field $k$, defined in Section~\ref{sec:GW}. We produce such a naive degree for endomorphisms of the projective line in Section~\ref{sec:lannes}, however in order to produce such a degree for smooth $n$-schemes in general, we will need to develop some machinery from $\A^1$-homotopy theory. A brief detour is taken to establish the setting in which one can study motivic spaces, defining the unstable motivic homotopy category in Section~\ref{sec:unstable-cat}, and establishing some basic, albeit important computations involving colimits of motivic spaces in Section~\ref{sec:colimits}. This discussion culminates in the purity theorem of Morel and Voevodsky, stated in Section~\ref{sec:purity}, which is requisite background for defining the local $\A^1$-degree following Morel.

In Section~\ref{sec:enum-geo}, we are finally able to define the local $\A^1$-degree of a morphism of schemes, which is a powerful, versatile tool in enriching enumerative geometry problems over arbitrary fields. At this point, we survey a number of recent results in $\A^1$-enumerative geometry. We discuss the Eisenbud--Khimshiashvili--Levine signature formula in Sections~\ref{sec:EKL} and \ref{sec:EKL-proof}, and we see its relation to the $\A^1$-degree, as proved in \cite{KW-EKL}. An enriched version of the $\A^1$-Milnor number is provided in Section~\ref{sec:Milnor-no}, which provides an enriched count of nodes on a hypersurface, following \cite{KW-Milnor-Number}. The problem of counting lines on a cubic surface, and the associated enriched results \cite{KW-arithmetic-count} are discussed in Section~\ref{sec:lines}. Finally, in Section~\ref{sec:four-lines} we provide an arithmetic count of lines meeting four lines in three-space, following \cite{Padma-Kirsten}.

Throughout these conference proceedings, various exercises (most of which were provided by Wickelgren in her 2018 lectures) are included. Detailed solutions may be found on the author's website.

\section*{Acknowledgements}
This expository paper is based around lectures by and countless conversations with Kirsten Wickelgren, who introduced me to this area of mathematics and has provided endless guidance and support along the journey. I am grateful to Mona Merling, who has shaped much of my mathematical understanding, and to Stephen McKean and Sabrina Pauli for many enlightening mathematical discussions related to $\A^1$-enumerative geometry. I am also grateful to Frank Neumann and Ambrus Pal for their work organizing these conference proceedings. Finally, I would like to thank the anonymous referee for their thoughtful feedback, which greatly improved this paper.

The author is supported by an NSF Graduate Research Fellowship, under grant number DGE-1845298. 

\section{Preliminaries}

\subsection{Enriching the topological degree}\label{sec:top-deg}

A continuous map $f: S^n \to S^n$ from the $n$-sphere to itself induces a homomorphism on the top homology group $f_\ast: H_n (S^n) \to H_n(S^n)$, which is of the form $f_\ast(x) = dx$ for some $d\in\Z$. This integer $d$ defines the \textit{(global) degree} of the map $f$. If $f$ and $g$ are homotopic as maps from the $n$-sphere to itself, they will induce the same homomorphism on homology groups. Therefore, taking $[S^n,S^n]$ to denote the set of homotopy classes of maps, we can consider degree as a function
\[
	\deg^\top: [S^{n},S^{n}] \to \Z.
\]
Throughout these notes, we will use the notation $\deg^\top$ to refer to the topological (integer-valued) degree. 

For any continuous map of $n$-manifolds $f: M\to N$, we could define a naive notion of the ``local degree'' around a point $p\in M$ via the following procedure: suppose that $q\in N$ has the property that $f^{-1}(q)$ is discrete, and let $p\in f^{-1}(q)$. Pick a small ball $W$ containing $q$, and a small ball $V \subseteq f^{-1}(W)$ satisfying $V \cap f^{-1}(q) = \left\{ p \right\}$. Then we may see that the spaces $V \big/(V \setminus \{p\}) \simeq (V\big/ \partial V) \simeq S^n$ are homotopy equivalent. Similarly, we have that $W\big/ (W \minus \{q\}) \simeq S^n$. We obtain the following diagram:
\begin{equation}\label{eqn:local-degree-topology}
	\begin{tikzcd}[ampersand replacement=\&]
	S^n \dar["\simeq" left]\ar[r, "g"] \& S^n \\ 
	V\big/(V \minus \{p\}) \rar["f" below] \& W \big/(W \minus \{q\}) \ar[u,"\simeq" right].
	\end{tikzcd}
\end{equation}
Thus we could define the local (topological) degree of $f$ around our point $p$, denoted $\deg^{\top}_p(f)$, to be the induced degree map on the $n$-spheres, that is, $\deg^\top_p (f) := \deg^\top(g)$ in the diagram above. If $f^{-1}(q) = \{p_1,\ldots,p_m\}$ is a discrete set of isolated points, we may relate the global degree to the local degree via the following formula
\[
	\deg^{\top}(f) = \sum_{i=1}^m \deg_{p_i}^\top(f).
\]
One may prove that the left hand side is independent of $q$, and thus that the choice of $q$ is arbitrary in calculating the global degree from local degrees. In differential topology, when discussing the degree of a locally differentiable map $f$ between $n$-manifolds, we have a simple formula for the local degree at a simple zero. We pick local coordinates $(x_1,\ldots,x_n)$ in a neighborhood of our point $p_i$, and local coordinates $(y_1,\ldots,y_n)$ around a regular value $q$. Then we can interpret $f$ locally as a map $f = (f_1,\ldots, f_n): \R^n \to \R^n$. Suppose that the Jacobian $\Jac(f)$ is nonvanishing at the point $p_i$. Then we define
\begin{align*}
    \deg^\top_{p_i}(f) = \sgn(\Jac(f)(p_i)) = \begin{cases} +1 & \text{if }\Jac(f)(q_i) > 0 \\ -1 & \text{if }\Jac(f)(q_i) < 0.
	\end{cases}
\end{align*}

When working over a field $k$, Barge and Lannes\footnote{Unpublished. See the note by Morel on \cite[p.1037]{Morel-ICM}.} defined a notion of degree for a map $\P^1_k \to \P^1_k$. Their insight was, rather than taking the sign of the Jacobian as in differential topology, to instead remember the value of $\Jac(f)(p_i)$ as a square class in $k^\times / \left( k^\times \right)^2$. Over the reals this recovers the sign, but over a general field we may have vastly more square classes. We encode this value as a rank one symmetric bilinear form over $k$, and we will soon see that this idea can be used to define a local degree at $k$-rational points, and that using field traces we can extend the definition of local degree to hold for points with residue fields a finite separable extension of $k$. These degrees, rather than being integers, are elements of the \textit{Grothendieck--Witt ring of $k$}, denoted $\GW(k)$, defined below.

\subsection{The Grothendieck--Witt Ring}\label{sec:GW}

Over a field $k$, we may form a semiring of isomorphism classes of non-degenerate symmetric bilinear forms (or quadratic forms if we assume $\char(k)\ne 2$) on vector spaces over $k$, using the operations $\otimes_k$ and $\oplus$. Group completing this semiring with respect to $\oplus$, we obtain the Grothendieck--Witt ring $\GW(k)$. For any $a\in k^\times$, we let $\gw{a} \in \GW(k)$ denote the following rank one bilinear form:
\begin{align*}
\gw{a}: k\times k & \to k \\
(x,y) & \mapsto axy.
\end{align*}

Symmetric bilinear forms are equivalent if they differ only by a change of basis. For example, if $b\neq 0$ we can see that $\gw{ab^2}(x,y) = \gw{a}(bx,by)$, so we identify $\gw{a}=\gw{ab^2}$ in $\GW(k)$, since these bilinear forms agree up to a vector space automorphism of $k$. We may describe $\GW(k)$ to be a ring generated by elements $\gw{a}$ for each $a \in k^\times \big/(k^\times)^2$, subject to the following relations \cite[Lemma~4.9]{morel}
\begin{enumerate}
\item $\gw{a}\gw{b} = \gw{ab}$
\item $\gw{a} + \gw{b} = \gw{ab(a+b)} + \gw{a+b}$, for $a+b \neq 0$
\item $\gw{a}+\gw{-a} = \gw{1} + \gw{-1}$. We conventionally denote this element as $\H := \gw{1} + \gw{-1}$, called the \textit{hyperbolic element} of $\GW(k)$.
\end{enumerate}

\begin{exercise} In the statements above, (1) and (2) imply (3).
\end{exercise}

\begin{proposition} We have a ring isomorphism $\GW(\C) \cong \Z$, given by taking the rank.
\end{proposition}
\begin{proof} We remark that $\left\langle a \right\rangle = \left\langle b \right\rangle$ for any $a,b\in \C^\times$, thus we only have one isomorphism class of non-degenerate symmetric bilinear forms in rank one.
\end{proof}
The isomorphism $\GW(\C) \cong \Z$ relates to a general fact that the $\A^1$-degree of a morphism of complex schemes recovers the size of the fiber, counted with multiplicity.

\begin{proposition} The rank and signature provide a group isomorphism $\GW(\R) \cong \Z\times \Z$.
\end{proposition}
\begin{proof} The Gram matrix of a symmetric bilinear form on $\R^n$ is an $n\times n$ real symmetric matrix $A$. After diagonalizing our matrix $A$, we can always find a change of basis in which the eigenvalues lie in the set $\{-1,0,1\}$. A non-degenerate symmetric bilinear form guarantees that no eigenvalues will vanish, so all of these eigenvalues will be $\pm 1$. We may define the \textit{signature} of $A$ as the number of 1's appearing on the diagonalized matrix minus the number of -1's, and by \textit{Sylvester's law of inertia} this determines an invariant on our matrix $A$. Thus we obtain an injective map
\begin{align*}
	\GW(\R) & \to \Z\times \Z \\
	A & \mapsto (\rank(A), \sig(A)).
\end{align*}
The image of this map is the subgroup $\left\{ (a+b,a-b) \ : \ a,b\in\Z \right\}$, which one may verify is isomorphic to $\Z \times \Z$.
\end{proof}

The multiplication on $\GW(\R)$ does not agree with that of $\Z \times\Z$, in the sense that $\GW(\R) \cong \Z \times\Z$ is not a ring isomorphism. However one may verify that the map
\begin{align*}
    \GW(\R) &\to \frac{\Z[t]}{(t^2-1)},
\end{align*}
given by sending $\left\langle 1 \right\rangle\mapsto 1$ and $\left\langle -1 \right\rangle\mapsto t$, is in fact a ring isomorphism, and hence provides the multiplicative structure of $\GW(\R)$.

\begin{proposition} The rank and determinant provide a group isomorphism $\GW(\F_q) \cong \Z \times \F_q^\times \big/ (\F_q^\times)^2$.
\end{proposition}
\begin{proof}[Proof~sketch] We may still use the rank of our matrix as an invariant for $\GW(\F_q)$. Additionally, we might use the determinant of our matrix to distinguish between symmetric bilinear forms. However note that, for any similar matrix $C^T A C$, it has determinant $\det(C^T A C) = \det(A)\det(C)^2$. Therefore, we can view the determinant as a well-defined map $\det:\GW(\F_q) \to \F_q^\times \big/ (\F_q^\times)^2$. After group completion, we obtain a map $\GW(\F_q) \xto{(\text{rank},\det)} \Z\times \F_q^\times \big/ (\F_q^\times)^2$, which we verify is a group isomorphism. For more details, see \cite[II,Theorem~3.5]{Lam}.
\end{proof}

One may use $\GW(\F_q)$ to understand $\GW(\Q_p)$ by applying the following result.

\begin{theorem}{\cite[VI,Theorem~1.4]{Lam}} \textit{(Springer's Theorem)} Let $K$ be a complete discretely valued field, and $\kappa$ be its residue field, with the assumption that $\char(\kappa) \neq 2$. Then there is an isomorphism of groups
\[
	\GW(K) \cong \frac{\GW(\kappa) \oplus \GW(\kappa)}{ \Z[\H,-\H] }.
\]
\end{theorem}

\begin{corollary} We have a group isomorphism $\GW(\C((t))) = \Z \oplus \Z\big/2$.
\end{corollary}

We should see how the Grothendieck--Witt ring interacts with extensions of fields. For a separable field extension $K\subset L$, and an element $\beta \in \GW(L)$, we can view the composition
\[
	V\times V\xto{\beta} L \xto{\Tr_{L/K}} K
\]
as an element of $\GW(K)$ by post-composing with the trace map $L\to K$, and considering $V$ as a $K$-vector space. This provides us a natural homomorphism between Grothendieck--Witt rings\footnote{When the field extension is assumed to be finite but the separability condition is dropped, a more general notion of transfer is given by \textit{Scharlau's transfer} \cite[VII~\S1]{Lam}.}
\begin{align*}
    \Tr_{L/K} : \GW(L) \to \GW(K).
\end{align*}
At the level of $\A^1$-homotopy theory, this trace comes from a transfer on stable homotopy groups --- for more detail see \cite[\S4]{morel}. Now that we have seen some computations involving the Grothendieck--Witt ring, we can develop in detail the notion of degree for maps of schemes.

\subsection{Lannes' formula}\label{sec:lannes}
Let $f: \P^1_k \to \P^1_k$ be a non-constant endomorphism of the projective line over a field of characteristic 0. We can then pick a rational point $q\in \P^1_k$, with fiber $f^{-1}(q) = \{p_1,\ldots, p_m\}$ such that $\Jac(f)(p_i) \neq 0$ for each $i$, where the Jacobian is computed by picking the same affine coordinates on both copies of $\P^1_k$. Since $\Jac(f)(p_i) \in k(p_i)$ is only defined in a residue field, we must precompose with the trace map in order to define the \textit{local $\A^1$-degree}
\begin{equation}\label{eqn:local-deg-P1}
\begin{aligned}
	\deg^{\A^1}_{p_i} f := \Tr_{k(p_i)/k} \gw{\Jac(f)(p_i)}.
\end{aligned}
\end{equation}

We can then define the \text{global $\A^1$-degree} of $f$ as the following sum, which is independent of our choice of rational point $q$ with discrete fiber (this fact is attributable to Lannes and Morel, although a detailed proof may be found in \cite[Proposition~14]{KW-EKL}):
\[
    \deg^{\A^1} f := \sum_{i=1}^m \Tr_{k(p_i)/k} \gw{\Jac(f)(p_i)}.
\]
\begin{exercise}\label{exerc:degree} Compute the $\A^1$-degrees of the following maps:
\begin{enumerate}
\item $\P_k^1 \to \P_k^1$, given by $z \mapsto az$, for $a\in k^\times$.
\item $\P_k^1 \to \P_k^1$, given by $z \mapsto z^2$.
\end{enumerate}
\end{exercise}

Maps of schemes $\P_k^1 \to \P_k^1$ are precisely rational functions $\frac{f}{g}$. Assuming that $f$ and $g$ are relatively prime, we can determine the classical topological (integer-valued) degree of this rational function as
\[
	\deg^\top\left(\frac{f}{g}\right) = \max\{ \deg^\top(f), \deg^\top(g)\}.
\]

To the rational function $\frac{f}{g}$, one may associated a bilinear form, called the \textit{B\'ezout form}, which is denoted $\Bez \left( \frac{f}{g} \right)$. This is done by introducing two variables $X$ and $Y$, and remarking that we have the following equality
\[
	\frac{f(X) g(Y) - f(Y)g(X)}{X-Y} = \sum_{1\leq i,j\leq n} B_{ij} X^{i-1}Y^{j-1},
\]
where $n=\deg^\top\left(\frac{f}{g}\right)$, and where $B_{ij} \in k$. We can see that this defines a symmetric bilinear form $k^n\times k^n \to k$, whose Gram matrix is given by the coefficients $B_{ij}$.

\begin{exercise} Compute the B\'ezout bilinear forms of the maps given in Exercise \ref{exerc:degree}.
\end{exercise}

\begin{theorem} \textit{(Cazanave)} We have that
\begin{align*}
    \Bez\left(\frac{f}{g}\right) = \deg^{\A^1}\left(\frac{f}{g}\right).
\end{align*}
This is stated in \cite[Theorem~2]{KW-rational-function}, but is attributable to \cite{Cazanave}.
\end{theorem}

This provides us with an efficient way to compute the $\A^1$-degree of rational maps while circumventing the tedium of computing the local $\A^1$-degree at each point in a fiber.

\subsection{The unstable motivic homotopy category}\label{sec:unstable-cat}

One of the primary ideas in $\A^1$-homotopy theory is to replace the unit interval in classical homotopy theory with the affine line $\A^1_k = \Spec(k[t])$. To this end, one might develop a \textit{naive $\A^1$-homotopy} of maps of schemes $f,g:X\to Y$ as a morphism
\[
	h: X\times \A_k^1 \to Y,
\]
such that $h(x,0) = f(x)$ and $h(x,1) = g(x)$ for all $x\in X$. This was first introduced by Karoubi and Villamayor \cite{Karoubi}. This notion of naive $\A^1$-homotopy is not generally the most effective, partially due to the following observation.

\begin{exercise}\cite{Asok-notes} Prove that naive $\A^1$-homotopy fails to be a transitive relation on hom-sets by considering three morphisms $\Spec k \to \Spec k[x,y]/(xy)$ identifying the points $(0,1)$, $(0,0)$, and $(1,0)$.
\end{exercise}

We will build a model category in which we have a class of $\A^1$\textit{-weak equivalences}, and we will denote by $[-,-]_{\A^1}$ the weak equivalence classes of morphisms. In particular, naive $\A^1$-homotopy equivalences are tractable examples of $\A^1$-weak equivalences. Nonetheless, naive $\A^1$-homotopy generates an equivalence relation, and in practice the naive homotopy classes of maps $[X,Y]_N$ are often easier to compute than their genuine counterparts $[X,Y]_{\A^1}$. In fact, the naive homotopy classes of maps $[\P_k^1,\P_k^1]_N$ are equipped with an addition, induced by pinch maps, which endows this set with a monoid structure. It was demonstrated by Cazanave that the map
\[
	[\P_k^1,\P_k^1]_N \to [\P_k^1,\P_k^1]_{\A^1}
\]
is a group completion \cite{Cazanave}.

In order to study the homotopy theory of schemes, we must develop a model structure which encodes a notion of $\A^1$-weak equivalence. In particular we must force $\A^1$ to be contractible --- as we have remarked, the initial motivation for forming such a model category was to treat $\A^1$ as if it were akin to the interval $[0,1]$ in the category of topological spaces. Morel and Voevodsky initially formulated the theory of the ``homotopy category of a site with an interval''; for this classical treatment see \cite[\S2.3]{MV99}.

We remark that the category of smooth $k$-schemes $\Sm_k$ does not admit all colimits, and therefore cannot be endowed with a model structure. To rectify this issue, we pass to the category of the simplicial presheaves via the Yoneda embedding
\begin{align*}
    \Sm_k &\to \sPre(\Sm_k) = \Fun(\Sm_k^\op, \sSet) \\
    X &\mapsto \Map(-, X).
\end{align*}
This new category is \textit{cocomplete} (it admits all small colimits), and moreover can be equipped with the \textit{projective model structure} arising from the classical model structure on simplicial sets. Given our model structure, we are now permitted to identify a class of morphisms which we would like to call weak equivalences, and perform Bousfield localization in order to formally invert them. For exposition on Bousfield localization and related results, we refer the reader to \cite{lawson2020introduction}.

The analog of open covers in a categorical setting is provided by a Grothendieck topology $\tau$. The category $\Sm_k$ can be equipped with a Grothendieck topology in order to make it a site, after which, we will apply Bousfield localization to render the class of $\tau$-hypercovers (our analog of open covers) into weak equivalences. We remark that by \cite[Theorem~6.2]{Dugger-Hollander}, this localization exists, and we denote it by $L_\tau: \sPre(\Sm_k) \to \Sh_{\tau,k}$. 
The fibrant objects in $\Sh_{\tau,k}$ are those simplicial presheaves which are \textit{homotopy sheaves} in the $\tau$ topology \cite[p.~20]{AE}. We therefore think about the localization $L_\tau$ as a way to encode the topology $\tau$ into the homotopy theory of $\sPre(\Sm_k)$.

Due to the wealth of properties granted to us by simplicial presheaves, the category $\Sh_{\tau,k}$ inherits a left proper combinatorial simplicial model category structure, and in particular we are allowed to perform Bousfield localization again in order to force $\A^1$ to be contractible. We identify a set of maps $\{X\times \A^1 \to X\}$, indexed over the set of isomorphism classes of objects in $\Sm_k$, as our desired weak equivalences, then perform a final Bousfield localization $L_{\A^1}$ with respect to this set. Finally, we define
\[
	\Spc^{\A^1}_{\tau,k} := L_{\A^1} \Sh_{\tau,k} = L_{\A^1} L_\tau \s\Pre(\Sm_k).
\]
This category has a model structure by construction, and we refer to its homotopy category as the \textit{unstable motivic homotopy category}. Throughout these notes and in much of the literature, it is assumed we are using the \textit{Nisnevich topology} (which is defined and contrasted with other choices of topologies below), and we will write $\Spc_k^{\A^1} := \Spc_{\Nis,k}^{\A^1}$. Our primary objects of study in $\Spc_k^{\A^1}$ will be the fibrant objects of this category, which we refer to as $\A^1$\textit{-spaces}. These admit a tangible recognition as precisely those presheaves which are valued in Kan complexes, satisfy Nisnevich descent, and are $\A^1$-invariant \cite[Remark~3.58]{AE}. For more detail, see \cite[\S3]{AE}.

There are many equivalent constructions of $\Spc^{\A^1}_k$, one notable one arising from the \textit{universal homotopy theory} on the category of smooth schemes, as described by \cite{Dugger-universal}. By freely adjoining homotopy colimits, we obtain a universal category $U(\Sm_k)$ which we may localize at the collections of maps
\begin{align*}
	&\left\{ \hocolim U_{\bullet} \to X \ : \ \{U_\alpha\} \text{ is a hypercover of $X$}\right\} \\
	&\{ X \times \A^1 \to X\}.
\end{align*}
This procedure produces a model category $U(\Sm_k)_{\A^1}$ which is Quillen equivalent to $\Spc_k^{\A^1}$.

\begin{remark} In more modern language, one may build $\Spc^{\A^1}_k$ using $(\infty,1)$-categories rather than model categories. Such a perspective may be found throughout the literature, for example in \cite{BachmannHoyois, Robalo}.
\end{remark}

One may study the categories $\Spc_{k,\tau}^{\A^1}$ arising from other choices of Grothendieck topologies, and indeed the homotopy theories arising from each selection behave quite differently and merit individual study. A small inexhaustive list of possible topologies includes the Zariski, Nisnevich, and \'etale topologies.

\begin{definition} Suppose that $X$ and $Y$ are smooth over a field $k$. Then we say $f:X\to Y$ is \textit{\'etale} at $x$ if the induced map on cotangent spaces
\[
	(f^\ast \Omega_{Y/k})_x \xto{\iso} \Omega_{X/k, x}
\]
is an isomorphism \cite[\S2.2,~Corollary~10]{Neron-models}. If we have the additional structure of coordinates on our base and target spaces, this is equivalent to the condition that $\Jac(f)\neq 0$ in $k(x)$.
\end{definition}
For example, any open immersion $X\ohook Y$ is a local isomorphism, and is therefore an \'etale map.

\begin{definition} Let $\left\{ f_\alpha : U_\alpha \to X \right\}$ be a family of \'etale morphisms. We say that it is
\begin{enumerate}
    \item an \textit{\'etale cover} if this is a cover of $X$, that is the underlying map of topological spaces is surjective
    \item a \textit{Nisnevich cover} if this is a cover of $X$, and for every $x\in X$ there exists an $\alpha \in A$ and $y\in U_\alpha$ such that $y\mapsto x$ and it induces an isomorphism on residue fields $k(y)\xto{\cong} k(x)$
    \item a \textit{Zariski cover} if this is a cover of $X$, and each $f_\alpha$ is an open immersion.
\end{enumerate}
\end{definition}

\begin{remark} Every Zariski cover is a Nisnevich cover, and every Nisnevich cover is an \'etale cover, however the converses of these statements do not hold.
\end{remark}

In the Nisnevich topology, we are also able to retain some of the advantages that the Zariski topology offers. One of the primary advantages is that algebraic $K$-theory satisfies Nisnevich descent. Additionally we are able to compute the Nisnevich cohomological dimension as the Krull dimension of a scheme \cite[p.94]{MV99}. Finally, we refer the reader to \cite[Proposition~7.2]{AE}, which allows us to treat morphisms of schemes locally as morphisms of affine spaces, analogous to charts of Euclidean space in differential topology.

\subsection{Colimits}\label{sec:colimits}

Recall that the primary motivation in passing from $\Sm_k$ to $\sPre(\Sm_k)$ was the existence of colimits. Despite the fact that $\Sm_k$ does not admit all small colimits, it still admits some --- as a class of examples, consider colimits of schemes arising from Zariski open covers. The problem is that the Yoneda embedding $y: \Sm_k \to \sPre(\Sm_k)$ does not preserve colimits in general, thus in our efforts to rectify the failure of $\Sm_k$ to admit colimits, we have essentially forgotten about the colimits that it did in fact possess. This is part of the motivation to localize at $\tau$-hypercovers --- we see that colimits of schemes correspond to hypercovers on the associated representable presheaves. By our discussion in the previous section, the localization $L_\tau$ can be considered as the localization precisely at the class of maps $\hocolim U_\bullet \to X$ for any $\tau$-hypercover $U_\bullet \to X$. Thus colimits of schemes are recorded in the category $\Spc_k^{\A^1}$ as homotopy colimits corresponding to hypercovers. For ease of reference, we summarize this in the following slogan.

\begin{slogan} Colimits of smooth schemes along $\tau$-covers yield homotopy colimits of motivic spaces.
\end{slogan}

To illustrate this point, we consider the following example, where $\G_m := \Spec k \left[ x, \frac{1}{x} \right]$ denotes the multiplicative group scheme.

\begin{example}\label{ex:P1-cover} Let $f:\G_m \to \A_k^1$ be given by $z\mapsto z$, and $g:\G_m \to \A_k^1$ be given by $z\mapsto \frac{1}{z}$. Then the diagram
\[
	\begin{tikzcd}
	\G_m\rar["f" above]\dar["g" left] & \A_k^1\dar \\
	\A_k^1\rar & \P_k^1\po 
	\end{tikzcd}
\]
is a homotopy pushout of motivic spaces.
\end{example}
\begin{proof} We see that the two copies of the affine line form a Zariski open cover of $\P_k^1$, and hence a Nisnevich open cover of schemes. This corresponds to a hypercover on the representable simplicial presheaves, and after localization at Nisnevich hypercovers, we see that the homotopy pushout of $\left( \A_k^1 \from \G_m \to \A_k^1 \right)$ is precisely $\P^1_k$.
\end{proof}

For based topological spaces, recall we have a smash product, defined as
\[
	X \smashprod Y = X\times Y \big/\left( (X\times\{y\})\cup(\{x\}\times Y)\right).
\]
We can think about the category of based topological spaces as the slice category $\ast/\Top$, where $\ast$ denotes the one-point space, i.e. the terminal object. By similarly taking the slice category under the terminal object $\ast := \Spec k$, we obtain a pointed version of $\Spc_k^{\A^1}$, which is often denoted by $\Spc_{k,\ast}^{\A^1}$.\footnote{We note that such a slice category must be taken at the level of model categories rather than homotopy categories in order to have a tractable pointed homotopy theory.}
We can then define the smash product as the homotopy cofiber of the canonical map between the coproduct of two pointed motivic spaces into their product:
\[
	\begin{tikzcd}
	X\vee Y\rar\dar & X\times Y\dar \\
	\ast\rar &  X\smashprod Y.\po
	\end{tikzcd}
\]
One may define the \textit{suspension} as $\Sigma X := S^1 \smashprod X$, which we may verify is the same as the homotopy cofiber of $X \to \ast$. One may see that, since $\A^1_k \simeq \Spec k$ is contractible, we have that Example~\ref{ex:P1-cover} implies that $\P_k^1$ is the homotopy cofiber of the unique map $\G_m \to \Spec k$. Concisely, this example tells us that
\begin{align*}
    \P^1_k \simeq \Sigma \G_m.
\end{align*}

Recall from topology that the spheres satisfy $S^n \smashprod S^m \cong S^{n+m}$. In developing a homotopy theory of schemes, we would like to search for a class of objects satisfying an analogous property. From this motivation, we uncover two types of spheres in $\Spc^{\A^1}_k$. The first, denoted $S^1$, is called the \textit{simplicial sphere}, and can be thought of as the union of three copies of the affine line, enclosing a triangle. As a simplicial presheaf, we think of it as the constant presheaf at $S^1 = \Delta^1/\partial \Delta^1$. Our second sphere, often called the \textit{Tate sphere}, is taken to be the projective line $\P_k^1\simeq S^1\smashprod \G_m$.

There are various conventions for the notation on spheres in $\A^1$-homotopy theory, and in the literature one may see $S^{p+q\alpha}$, $S^{p,q}$ or $S^{p+q,q}$ to mean the same thing, depending on the context. In these notes, we will use the convention that
\[
	S^{p+q\alpha}:= (S^1)^{\smashprod p} \smashprod (\G_m)^{\smashprod q}.
\]

\begin{exercise} Show that the diagram
\[
	\begin{tikzcd}
	X\times Y\rar\dar & X\dar \\
	Y\rar & \Sigma (X\smashprod Y)\po
	\end{tikzcd}
\]
is a homotopy pushout diagram. The context for this example is left ambiguous as the result holds in $\Spc^{\A^1}_{k,\ast}$ just as well as it does for pointed topological spaces.
\end{exercise}

\begin{example}\label{ex:affine-space-minus-0} There is an $\A^1$-homotopy equivalence $\A_k^n \minus\{0\} \simeq (S^1)^{\smashprod (n-1)} \smashprod (\G_m)^{\smashprod n}$.
\end{example}
\begin{proof} Note that we may construct $\A_k^n\minus\{0\}$ as a homotopy pushout
\[
	\begin{tikzcd}
	(\A_k^1\minus\{0\})\times(\A_k^{n-1}\minus\{0\})\rar\dar & \A_k^1\times(\A_k^{n-1}\minus\{0\})\dar \\
	(\A_k^1 \minus \{0\})\times \A_k^n\rar & \A_k^n\minus \left\{ 0 \right\} \po.
	\end{tikzcd} 
\]
Applying the exercise above, we see that 
\begin{align*}
	\A_k^n \minus\{0\} &\simeq \Sigma (\A_k^{n-1} \minus\{0\})\smashprod (\A_k^1 \minus \{0\}) = S^1 \smashprod (\A_k^{n-1} \minus\{0\}) \smashprod \G_m.
\end{align*}
The result follows inductively.
\end{proof}

\begin{notation} For a morphism of motivic spaces $f: X \to Y$, denote by $Y/X$ the homotopy cofiber of the map $f$, that is, the homotopy pushout
\[ \begin{tikzcd}
    X\rar["f" above]\dar & Y\dar\\
    \ast\rar & Y/X.\po\\
\end{tikzcd} \]
\end{notation}

\begin{example}\label{ex:excision} \textit{(Excision)} Suppose that $X$ is a smooth scheme over $k$, that $Z \clhook X$ is a closed immersion, and that $U \supseteq Z$ is a Zariski open neighborhood of $Z$ inside of $X$. Then we have a Nisnevich weak equivalence (that is, a weak equivalence in the category $\Sh_{\Nis,k}$)
\begin{align*}
    \frac{U}{U\minus Z} \xto{\sim} \frac{X}{X \minus Z}.
\end{align*}
We refer to this result informally as \textit{excision} (not to be confused with excision in the sense of \cite[Proposition~3.53]{AE}), as we regard this weak equivalence as excising the closed subspace $X \minus U$ from the top and bottom of the cofiber $X/(X\minus Z)$.
\end{example}
\begin{proof} We remark that $(X\minus Z)$ and $U$ form a Zariski open cover of $X$, and that their intersection is $(X\minus Z) \cap U = U\minus Z$. As Zariski covers are Nisnevich covers, one remarks that we have a homotopy pushout diagram of motivic spaces
\[ \begin{tikzcd}
    U\minus Z\rar\dar & X\minus Z\dar\\
    U\rar & X.\po\\
\end{tikzcd} \]
The fact that the homotopy cofibers of the vertical maps in the diagram above are $\A^1$-weakly equivalent follows from the following diagram: 
\[ \begin{tikzcd}
    U\minus Z\rar\dar & X\minus Z\dar\rar & \ast\dar\\
    U\rar & X\rar\po & \frac{X}{X\minus Z}.\po\\
\end{tikzcd} \]
As the left and right squares are homotopy cocartesian, it follows formally that the entire rectangle is homotopy cocartesian.
\end{proof}

\begin{example}\label{ex:cofiber-proj-space} There is an $\A^1$-homotopy equivalence $\P_k^n \big/ \P_k^{n-1} \simeq (S^1)^{\smashprod n}\smashprod (\G_m)^{\smashprod n}$
\end{example}
\begin{proof} As $\P^n_k \minus \left\{ 0 \right\}$ is the total space of $\O(1)$ on $\P^{n-1}_k$, we have an $\A^1$-equivalence $\P^n_k \minus \left\{ 0 \right\} \simeq \P^{n-1}_k$. Therefore, one sees $\P_k^n \big/ \P_k^{n-1} \simeq \P_k^n \big/ (\P_k^n - \{0\})$. Via excision, we are able to excise everything away from a standard affine chart, from which we may see that $\P_k^n \big/ (\P_k^n - \{0\}) \simeq \A_k^n \big/(\A_k^n - \{0\})$. Contracting $\A_k^n$, we obtain $\ast\big/(\A_k^n - \{0\}) \simeq \Sigma (\A_k^n - \{0\})$. Therefore $\P_k^n \big/ \P_k^{n-1} \simeq \Sigma(\A_k^n - \{0\}) \simeq (S^1)^{\smashprod n}\smashprod (\G_m)^{\smashprod n}$ after applying Example~\ref{ex:affine-space-minus-0}.
\end{proof}

This last example is of particular interest, as it exhibits the cofiber $\P_k^n \big/ \P_k^{n-1}$ as a type of sphere in $\A^1$-homotopy theory. Given an endomorphism of such a motivic sphere, Morel defined a \textit{degree homomorphism}
\begin{align*}
    \deg^{\A^1}: \left[ \P^n_k/\P^{n-1}_k , \P^n_k/ \P^{n-1}_k \right]_{\A^1} \to \GW(k),
\end{align*}
which he proved was an isomorphism in degrees $n\ge 2$ \cite[Corollary~4.11]{Morel-ICM}.

Recall that to define a local Brouwer degree of an endomorphism between $n$-manifolds, we first had to pick a ball containing a point $p$, and then identify the cofiber $W/(W\minus \left\{ p \right\})$ with the $n$-sphere $S^n$. This allowed us to construct Diagram~\ref{eqn:local-degree-topology}, after which we could apply the degree homomorphism $[S^n,S^n] \to \Z$ to define a local degree. An analogous procedure will be available to us in $\A^1$-homotopy theory if, for a Zariski open neighborhood $U$ around a $k$-rational point $x$, we are able to associate a canonical $\A^1$-weak equivalence between $U/(U\minus\{x\})$ and $\P^n_k/\P^{n-1}_k$. Indeed this is possible via the theorem of \textit{purity}.

\subsection{Purity}\label{sec:purity}

One of the major techniques in $\A^1$-homotopy theory comes from the purity theorem. In manifold topology, the tubular neighborhood theorem allows us to define a diffeomorphism between a tubular neighborhood of a smooth immersion and an open neighborhood around its zero section in the normal bundle. In $\A^1$ homotopy theory, the Nisnevich topology isn't fine enough to define such a tubular neighborhood, however we can still get an analog of the tubular neighborhood theorem which will allow us to define, among other things, local $\A^1$-degrees of maps.

\begin{definition} A \textit{Thom space} of a vector bundle $V \to X$ is the cofiber
\[
	V \big/ (V\minus X),
\]
where $V\minus X$ denotes the vector bundle minus its zero section. In the literature, this may be denoted by
\[
	\text{Thom}(V,X) = \Th(V) = X^V.
\]
\end{definition}

\begin{remark} We may also describe the Thom space of a vector bundle via an $\A^1$-weak equivalence
\[ \Th(V) \simeq \frac{\Proj(V \oplus \O)}{\Proj(V)}.\]
\end{remark}
\begin{proof} We have a map $V \to V \oplus \O$ sending $v \mapsto (v,1)$, and we may view this inside of projective space via the inclusion $V \oplus \O \subseteq \Proj \left( V \oplus \O \right)$. Via excision (Example~\ref{ex:excision}), we have a Nisnevich weak equivalence
\begin{align*}
    \frac{\Proj(V \oplus \O)}{\Proj(V \oplus \O) \minus 0} \simeq \frac{V}{V\minus 0},
\end{align*}
where $0$ denotes the image of the zero section. We remark that $\Proj(V\oplus\O)\minus 0$ is the total space of $\O(-1)$ on $\Proj(V)$, thus we have an $\A^1$-weak equivalence $\Proj(V \oplus\O)\minus 0 \simeq \Proj(V)$. The result follows from observing $\frac{\Proj(V \oplus \O)}{\Proj(V \oplus \O) \minus 0} \simeq \frac{\Proj(V\oplus\O)}{\Proj(V)}$.
\end{proof}

\begin{theorem} \textit{(Purity theorem)} Let $Z\clhook X$ be a closed immersion in $\Sm_k$. Then we have an $\A^1$-equivalence
\[
	\frac{X}{X\minus Z} \simeq \Th(N_Z X),
\]
where $N_Z X \to Z$ denotes the normal bundle of $Z$ in $X$.
\end{theorem}
\begin{proof} The proof uses the deformation to the normal bundle of Fulton and MacPherson \cite{Fulton}. Let $f$ denote the composition of the maps
\begin{align*}
	\Bl_{Z\times \{0\}} (X\times \A_k^1) \to X\times \A_k^1 \to \A_k^1.
\end{align*}

We define $D_Z X$ to be the scheme $\Bl_{Z\times\{0\}} (X\times \A_k^1) \minus \Bl_{Z\times\{0\}} (X\times\{0\})$, and note that $f$ restricts to a map $f\Big|_{D_Z X} : D_Z X \to \A_k^1$. We may compute the fiber of $\left. f \right|_{ D_Z X }$ over 0 as 
\begin{align*}
	f\Big|_{D_Z X}^{-1}(0) &= \Proj \left( N_{Z\times\{0\}} (X\times \A_k^1) \right) \minus \Proj \left( N_{Z\times\{0\} } (X\times \{0\})  \right) \\
	&= \Proj (N_Z X \oplus \O) \minus \Proj \left( N_Z X \right) \\
	&= N_Z X,
\end{align*}
and the fiber over 1 as $f\Big|_{D_Z X}^{-1}(1) = X$. Since $Z\times \A_k^1$ determines a closed subscheme in $D_Z X$, we have that the fiber over 0 is $Z\subseteq N_Z X$ and the fiber over 1 is $Z\subseteq X$. Thus we obtain morphisms of pairs

\begin{equation}\label{eqn:normal-bundle-purity}
\begin{aligned}
    (Z, N_Z X) &\xto{i_0} (Z\times \A_k^1, D_Z X) \\
	(Z,X) &\xto{i_1} (Z\times \A_k^1, D_Z X),
\end{aligned}
\end{equation}

corresponding to the inclusions of the fibers over the points 0 and 1, respectively. To prove the purity theorem, it now suffices to show that the induced morphisms on cofibers are weak equivalences:
\begin{align*}
	\frac{N_Z X}{N_Z X \minus Z} &\to \frac{D_Z X}{D_Z X \minus Z\times \A_k^1} \\
	\frac{X}{X\minus Z} &\to \frac{D_Z X}{D_Z X \minus Z\times \A_k^1}.
\end{align*}

\begin{lemma}\label{lem:smooth-pairs}\cite[Lemma 7.3]{AE} Suppose that \textbf{P} is a property of smooth pairs of schemes such that the following properties hold:
\begin{enumerate}
\item If $(Z,X)$ is a smooth pair of schemes and $\{U_\alpha \to X\}_{\alpha \in A}$ is a Zariski cover of $X$ such that \textbf{P} holds for the pair 
\[
	(Z\times_X U_{\alpha_1}\times_X \cdots \times_X U_{\alpha_n}, U_{\alpha_1}\times_X \cdots \times_X U_{\alpha_n})
\]
for each $(\alpha_1,\ldots,\alpha_n)$, then \textbf{P} holds for $(Z,X)$
\item If $(Z,X)\to (Z,Y)$ is a morphism of smooth pairs inducing an isomorphism on $Z$ such that $X\to Y$ is \'etale, then \textbf{P} holds for $(Z,X)$ if and only if \textbf{P} holds for $(Z,Y)$
\item \textbf{P} holds for the pair $(Z,\A_k^n\times Z)$,
\end{enumerate}
then \textbf{P} holds for all smooth pairs.
\end{lemma}

To conclude the proof of purity, we let \textbf{P} be the property on the pair $(Z,X)$ that the morphisms in Equation \ref{eqn:normal-bundle-purity} induce homotopy pushout diagrams\footnote{Equivalently, one may say that $i_0$ and $i_1$ are \textit{weakly excisive} morphisms of pairs \cite[Definition~3.17]{Hoyois-six}.}
\[ \begin{tikzcd}
    Z\rar\dar & \frac{N_Z X}{N_Z X \minus Z}\dar \\
    Z\times \A^1_k\rar & \frac{D_Z X}{D_Z X \minus Z\times \A^1_k}\po
\end{tikzcd} \qquad \begin{tikzcd}
    Z\rar\dar & \frac{X}{X \minus Z}\dar \\
    Z\times \A^1_k\rar & \frac{D_Z X}{D_Z X \minus Z\times \A^1_k}\po.
\end{tikzcd} \]

One may check that Lemma \ref{lem:smooth-pairs} holds for this property, and therefore since $Z \to Z\times \A^1_k$ is a weak equivalence, a homotopy pushout along this map is also a weak equivalence. Thus we obtain a sequence of $\A^1$-weak equivalences
\begin{align*}
    \frac{X}{X\minus Z} \xto{\sim} \frac{D_Z X}{D_Z X \minus Z\times \A^1_k} \xfrom{\sim} \frac{N_Z X}{N_Z X \minus Z} = \Th(N_Z X).
\end{align*}
\end{proof}

\section{$\A^1$-enumerative geometry}\label{sec:enum-geo}
As discussed above, Morel exhibited the global degree of maps of motivic spheres as
\[
	\deg^{\A^1} : [\P_k^n/\P_k^{n-1}, \P_k^n/\P_k^{n-1}]_{\A^1} \to \GW(k).
\]
Recall that, for a scheme $X$, we have functors to the category of topological spaces obtained by taking real and complex points, that is, $X\mapsto X(\R)$ and $X \mapsto X(\C)$. Morel's degree map satisfies a compatibility diagram with the degree maps we recognize from algebraic topology\footnote{The commutativity of this diagram is one of the key features of Morel's $\A^1$-degree and is attributable to him \cite[p.~1037]{Morel-ICM}. We can provide an alternative justification of this fact following the discussion of the EKL form in Section~\ref{sec:EKL}.}

\begin{equation}\label{fig:compatibility-diagram}
\begin{tikzcd}[ampersand replacement=\&]
	{[S^n,S^n]}\dar["\deg^\top" left] \& {[\P_\R^n/\P_\R^{n-1}, \P_\R^n/\P_\R^{n-1}]_{\A^1}}\rar["\C\text{-pts}"] \lar["\R\text{-pts}" above]\dar["\deg^{\A^1}"] \& {[S^{2n},S^{2n}]}\dar["\deg^\top" right] \\
	\Z \& \GW(\R)\rar["\text{rank}" below]\lar["\text{sig}" below] \& \Z. 
	\end{tikzcd}
\end{equation}

We can apply the purity theorem to develop a notion of local degree for a general map between schemes of the same dimension. Suppose that $f:\A_k^n \to \A_k^n$, and $x\in\A_k^n$ is a $k$-rational preimage of a $k$-rational point $y = f(x)$. Further suppose that $x$ is an isolated point in $f^{-1}(y)$, meaning that there exists a Zariski open set $U\subseteq \A_k^n$ such that $x\in U$ and $f^{-1}(y) \cap U = x$.

\begin{definition}\label{def:local-a1-deg} In the conditions above, the \textit{local $\A^1$-degree} of $f$ at $x$ is defined to be the degree of the map
\[
	U\Big/ (U\minus \{x\}) \xto{\overline{f}} \A_k^n \Big/(\A_k^n \minus\{y\}),
\]
under the $\A^1$-weak equivalences $U\Big/ (U\minus \{x\}) \cong \Th(T_x \A_k^n) \cong \P_k^n\Big/\P_k^{n-1}$ and $\A_k^n \Big/(\A_k^n \minus\{y\}) \cong \P_k^n \Big/\P_k^{n-1}$ provided to us by purity and by the canonical trivialization of the tangent space of affine space.
\end{definition}

Dropping the assumption that $k(x)=k$, but still assuming that $y$ is $k$-rational, we may equivalently define $\deg^{\A^1}_x f$ as the degree of the composite
\begin{align*}
	 \P_k^n\Big/\P_k^{n-1} \to  \P_k^n\Big/(\P_k^{n}\minus\{x\})\cong U\Big/ (U\minus \{x\}) \xto{\overline{f}} \A_k^n \Big/(\A_k^n \minus\{y\}) \cong \P_k^n\Big/\P_k^{n-1}.
\end{align*}

\begin{proposition} These definitions of the local degree are equivalent. This was proven in \cite[Prop.~12]{KW-EKL}, which is a generalization of a proof of Hoyois \cite[Lemma~5.5]{Hoyois-lefschetz}.
\end{proposition}

Equation~\ref{eqn:local-deg-P1} admits the following generalization to endomorphisms of affine space.
\begin{proposition}\label{prop:local-deg-tr-Jac}\cite[Proposition~15]{KW-EKL} Let $f: \A^n_k \to \A^n_k$, assume that $f$ is \'etale at a closed point $x \in \A^n_k$, and assume that that $f(x) = y$ is $k$-rational and that $x$ is isolated in its fiber. Then the local degree is given by
\begin{align*}
    \deg_x^{\A^1}(f) = \Tr_{k(x)/k} \left\langle \Jac(f)(x) \right\rangle.
\end{align*}
\end{proposition}

\begin{remark} At a non-rational point $p$ whose residue field $k(p)|k$ is a finite separable extension of the ground field, the local $\A^1$-degree can be computed by base changing to $k(p)$ to compute the local degree rationally, and applying the field trace $\Tr_{k(p)/k}$ to obtain a well-defined element of $\GW(k)$ \cite{aws-paper}.
\end{remark}

\subsection{The Eisenbud--Khimshiashvili--Levine signature formula}\label{sec:EKL}
Given a morphism $f = (f_1, \ldots, f_n): \A^n_k \to \A^n_k$ with an isolated zero at the origin, we may associate to it a certain isomorphism class of bilinear forms $w_0^\EKL(f)$, called the \textit{Eisenbud--Levine--Khimshiashvili (EKL) class}. This was studied by Eisenbud and Levine, and independently by Khimshiashvili, in the case where $f$ is a smooth endomorphism of $\R^n$ \cite{Eisenbud-Levine,Khimshiashvili}. They ascertained that the degree $\deg^\top_0 f$ can be computed as the signature of the form $w_0^\EKL(f)$. If $f$ is furthermore assumed to be real analytic, the rank of this form recovers the degree of the complexification $f_\C$ \cite{Pal67}. This bilinear form $w_0^\EKL(f)$ can be defined over an arbitrary field $k$, and in this setting Eisenbud asked the following question: does $w_0^\EKL(f)$ have any topological interpretation? We will see that the answer is yes, via work of Kass and Wickelgren \cite{KW-EKL}.

Suppose that $f = (f_1, \ldots, f_n): \A^n_k \to \A^n_k$ has an isolated zero at the origin, and define the local $k$-algebra
\begin{align*}
    Q_0(f) := \frac{k[x_1, \ldots, x_n]_{(x_1, \ldots, x_n)}}{(f_1, \ldots, f_n)}.
\end{align*}

We may pick polynomials $a_{ij}$ so that, for each $i$, we have the equality
\begin{align*}
    f_i(x_1, \ldots, x_n) &= f_i(0) + \sum_{j=1}^n a_{ij} \cdot x_j.
\end{align*}
By taking their determinant, we define $E_0(f) := \det(a_{ij})$ as an element of $Q_0(f)$, which we refer to as the \textit{distinguished socle element} of the local algebra $Q_0(f)$. We remark that when $\Jac(f)$ is a nonzero element of $Q_0(f)$, one has the equality \cite[4.7~Korollar]{SchejaStorch}
\begin{align*}
    \Jac(f) = \dim_k(Q_0(f)) \cdot E_0(f).
\end{align*}

We then pick $\eta$ to be any $k$-linear vector space homomorphism $\eta: Q_0(f) \to k$ satisfying $\eta(E_0(f)) = 1$. One may check that the following bilinear form
\begin{align*}
    Q_0(f) \times Q_0(f) &\to k \\
    (u,v) &\mapsto \eta(u\cdot v)
\end{align*}
is non-degenerate and its isomorphism class is independent of the choice of $\eta$ \cite[Propositions~3.4,3.5]{Eisenbud-Levine}, \cite[\S3]{KW-EKL}. The class of this form in $\GW(k)$ is referred to as the EKL class, and denoted by $w_0^\EKL(f)$.

\begin{example} If $f: \A_k^1 \to \A_k^1$ is given by $z\mapsto z^2$, we may see that $Q_0(f) = k[z]_{(z)}\Big/(z^2)$. We see that $f$ has an isolated zero at the origin, and that
\begin{align*}
    f = f(0) + x\cdot x,
\end{align*}
hence $E_0(f) = x$. We determine $\eta: Q_0(f)\to k$ on a basis for $Q_0(f)$ by setting $\eta(x) =1$ and $\eta(1) =0$. Then we compute the EKL form via its Gram matrix as:
\begin{align*}
    \begin{pmatrix} \eta(1\cdot 1) & \eta(1\cdot x) \\ \eta(x\cdot 1) & \eta(x\cdot x) \end{pmatrix} &= \begin{pmatrix} 0 & 1 \\ 1 & 0 \end{pmatrix} = \H.
\end{align*}
\end{example}

\begin{theorem}\label{kw-deg-ekl} If $f:\A^n_k \to \A^n_k$ is any endomorphism of affine space with an isolated zero at the origin, there is an equality $\deg^{\A^1}_0 f = w_0^{\EKL}(f)$ in $\GW(k)$ \cite{KW-EKL}.
\end{theorem}

In particular we observe that the compatibility stated in Diagram~\ref{fig:compatibility-diagram} is justified by this theorem, combined with the results of Eisenbud--Khimshiashvili--Levine and Palamodov. Moreover we remark that the EKL form can be defined at any $k$-rational point, and an analogous statement to Theorem~\ref{kw-deg-ekl} holds in this context.

\begin{exercise}\label{exercise-wekl} $\ $
\begin{enumerate}

	\item Compute the degree of $f : \A_k^2 \to \A_k^2$, given as $f(x,y) = (4x^3, 2y)$ in the case where $\char(k)\neq 2$.
	
	\item Supposing $f$ is \'etale at the origin 0, show that $w_0^\EKL(f) = \left\langle \Jac(f)(0) \right\rangle$ is an equality in $\GW(k)$. Show furthermore that an analogous equality holds at any $k$-rational point $x$.
\end{enumerate}
\end{exercise}
As a generalization of Exercise~\ref{exercise-wekl}(2), one may show that if $f$ is \'etale at a point $x$, one has the following equality in $\GW(k)$
\begin{equation}\label{eqn:etale-EKL-local-deg}
\begin{aligned}
    w_x^\EKL(f) = \Tr_{k(x)/k} \left\langle \Jac(f)(x) \right\rangle.
\end{aligned}
\end{equation}
This is shown using Galois descent, as in \cite[Lemma~33]{KW-EKL}.

\subsection{Sketch of proof for Theorem~\ref{kw-deg-ekl}}\label{sec:EKL-proof}

\textit{Step 1}: We can see that $\deg^{\A^1}_0 f$ and $w_0^\EKL(f)$ are finitely determined in the sense that they are unchanged by changing $f$ to $f+g$, with $g = (g_1, \ldots, g_n)$, and $g_i \in \mathfrak{m}_0^N$ for sufficiently large $N$, where $\mathfrak{m}_0:= (x_1, \ldots, x_n)$ denotes the maximal ideal at the origin \cite[Lemma~17]{KW-EKL}.

\textit{Step 2}: By changing $f$ to $f+g$, we may assume that $f$ extends to a finite, flat morphism $F:\P_k^n \to \P_k^n$, where $F^{-1}(\A_k^n) \subseteq \A_k^n$ and $\left. F\right|_{F^{-1}(0) \minus \{ 0 \}}$ is \'etale \cite[Proposition~23]{KW-EKL}.

\begin{proposition}\label{prop:SchejaStorch} \textit{(Scheja-Storch)} \cite[\S3, pp.180---182]{SchejaStorch} We have that $w_0^\EKL(f)$ is a direct summand of the fiber at 0 of a family of bilinear forms over $\A_k^n$, which we construct below.
\end{proposition}

We will prove Proposition~\ref{prop:SchejaStorch} following the construction of this family of bilinear forms.

\begin{customenvironment}{The Scheja-Storch construction} Let $F: \Spec(P) \to \Spec(A)$, where
\begin{align*}
	P &= k[x_1, \ldots, x_n]\\
	A &= k[y_1, \ldots, y_n].
\end{align*}

One may show that the collection $\{t_1,\ldots, t_n\}$ is a regular sequence in $A[x_1, \ldots, x_n]$, where $t_i := y_i - F_i (x_1, \ldots, x_n)$. Then
\begin{align*}
	B = A[x_1, \ldots, x_n] \big/ \lrangle{ t_1, \ldots, t_n}
\end{align*}
is a relative complete intersection, which parametrizes the fibers of $F$. This regular sequence determines a canonical isomorphism \cite[Satz 3.3]{SchejaStorch}
\begin{align*}
	\theta : \Hom_A(B,A) \xto{\cong} B,
\end{align*}
via the following procedure: we may first express
\begin{align*}
	t_j \otimes 1 - 1 \otimes t_j = \sum_{i=1}^n a_{ij} \left( x_i \otimes 1 - 1 \otimes x_i \right),
\end{align*}
where each $a_{ij}$ is an element of $A[x_1, \ldots, x_n] \otimes_A A[x_1, \ldots, x_n]$. Under the projection map $A[x_1,\ldots, x_n] \otimes_A A[x_1,\ldots, x_n] \to B\otimes_A B$, we have that $\det\left( a_{ij} \right)$ is mapped to some element $\Delta$. We now consider the bijection
\begin{align*}
	B \otimes_A B &\to \Hom_A( \Hom_A(B,A), B) \\
	b \otimes c &\mapsto \left( \phi \mapsto \phi(b) c \right),
\end{align*}

and define $\theta$ to be the image of $\Delta$. We remark that a priori $\theta$ is an $A$-module homomorphism between $\Hom_A(B,A)$ and $B$, which both have $B$-module structures. It is in fact a $B$-module homomorphism, and is moreover an isomorphism by \cite[Satz~3.3]{SchejaStorch}. Defining $\eta = \theta^{-1}(1)$, we have that $\eta$ determines a bilinear form, which we denote by $w$
\begin{align*}
	B \otimes_A B &\to A \\
	b \otimes c &\xmapsto{w} \eta(bc).
\end{align*}
\end{customenvironment}

\textit{Proof of Proposition \ref{prop:SchejaStorch}}: We note that, when $y_1 = \ldots = y_n = 0$, $\Spec(B)$ is the fiber of $F$ over 0, consisting of a discrete set of points. This corresponds to a disjoint union of schemes. If $b$ and $c$ lie in different components, then their product is zero. This implies that the bilinear form $w$ decomposes into an orthogonal direct sum of forms over each factor in $F^{-1}(0)$. These factors correspond to EKL forms at each point in the fiber $F^{-1}(0)$, and in particular over $0 \in F^{-1}(0)$, we recover the EKL form $w_0^\EKL(F)$.

\hfill\qed

The following theorem will allow us to relate the EKL forms at various points in the fiber $F^{-1}(0)$.

\begin{theorem} \textit{(Harder's Theorem)} \cite[VII.3.13]{Lam-serre-conj} A family of symmetric bilinear forms over $\A^1_k$ is constant (respectively, has constant specialization to $k$-points) for characteristic not equal to 2 (resp. any $k$). In particular when $\char(k) \ne 2$, for any finite $k[t]$-module $M$, we have that the family of bilinear forms $M\times_{k[t]} M \to k[t]$ is pulled back from some bilinear form $N\times_k N \to k$ via the unique morphism of schemes $\A^1_k \to \Spec(k)$.
\end{theorem}

\textit{Step 3}: We choose $y$ so that $\left. F \right|_{F^{-1} (y)}$ is \'etale. One may use the generalization of Exercise~\ref{exercise-wekl}(2) as stated in Equation~\ref{eqn:etale-EKL-local-deg}, combined with Proposition~\ref{prop:local-deg-tr-Jac} to see that
\begin{align*}
	\sum_{x\in F^{-1}(y)} w_x^\EKL(F) = \sum_{x \in F^{-1} (y)} \deg^{\A^1}_x F.
\end{align*}

By Harder's theorem, we have that $\sum_{x\in F^{-1}(y)} w_x^\EKL(F) = \sum_{x\in F^{-1}(0)} w_x^\EKL(F)$, and by the local formula for degree, we see that
\begin{align*}
	\sum_{x \in F^{-1} (y)} \deg^{\A^1}_x F = \deg^{\A^1} F = \sum_{x\in F^{-1}(0)} \deg^{\A^1}_x F.
\end{align*}

Thus $\sum_{x\in F^{-1}(0)} w_x^\EKL(F) = \sum_{x \in F^{-1} (0)} \deg^{\A^1}_x F$. Since $\left. F \right|_{F^{-1}(0) \minus \{ 0 \} }$ is \'etale, we may iteratively apply the equality in Equation~\ref{eqn:etale-EKL-local-deg} to cancel terms, leaving us with the local degree and EKL form at the origin:
\begin{align*}
	w_0^\EKL(F) = \deg^{\A^1}_0 F.
\end{align*}
Therefore by finite determinacy we recover the desired equality $w_0^\EKL(f) = \deg_0^{\A^1}(f)$. This concludes the proof of Theorem \ref{kw-deg-ekl}.

\subsection{$\A^1$-Milnor numbers}\label{sec:Milnor-no}

The following section is based off of joint work by Jesse Kass and Kirsten Wickelgren \cite[\S8]{KW-EKL}. A variety over a perfect field is generically smooth, although it may admit a singular locus where the dimension of the tangent space exceeds the dimension of the variety, for example a self-intersecting point on a singular elliptic curve. Singularities are generally difficult to study, although certain classes are more tractable than others. There is a particular class of singularities, called \textit{nodes}, which are in some sense the most generic. If $k$ is a field of characteristic not equal to 2, then a node is given by an equation $x_1^2 + \ldots + x_n^2 = 0$ over a separable algebraic closure $\bar{k}$.

Consider a point $p$ on a hypersurface $\{f(x_1,\ldots,x_n) = 0\} \subseteq \A_k^n$. Fix values $a_1, \ldots, a_n$, and consider the family
\begin{align*}
	f(x_1, \ldots, x_n) + a_1 x_1 + \ldots + a_n x_n = t,
\end{align*}
parametrized over the affine $t$-line. This hypersurface bifurcates into nodes over $\bar{k}$. Given any hypersurface $g(x_1, \ldots, x_n)$ with a node at a $k$-rational point $p$, we define the \textit{type} of the node as the element in $\GW(k)$ corresponding to the rank one form represented by the Hessian matrix at $p$:
\begin{align*}
    \type(p) := \left\langle \frac{\partial^2 g}{\partial x_i \partial x_j}(p) \right\rangle.
\end{align*}
In particular, we see that:
\begin{align*}
	\type(x_1^2 + a x_2^2 = 0) &:= \gw{a} \\
	\type\left( \sum_{i=1}^n a_i x_i^2 = 0\right) &:= \gw{ 2^n \prod_{i=1}^n a_i}.
\end{align*}
In the case where we have a node at $p$ with $k(p) = L$, then $L$ is separable over $k$ \cite[Expos\'e XV, Th\'eor\`eme 1.2.6]{Deligne-Katz}, and we define the type of the node as the trace of the type over its residue field. In the examples above, this gives:
\begin{align*}
	\type\left( \sum_{i=1}^n a_i x_i^2 = 0\right) := \Tr_{L/k}\gw{ 2^n \prod_{i=1}^n a_i}.
\end{align*}
Thus the type encodes the field of definition of the node, as well as its tangent direction. In the case where $k=\R$, we can visualize the possible $\R$-rational nodes in degree two as:

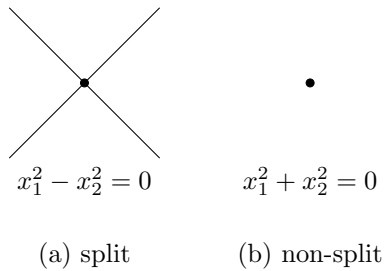
\begin{figure}[H]
	\begin{tikzpicture}
        \draw (0,2) -- (2,0);
		\draw (0,0) -- (2,2);
	    \filldraw (1,1) circle[radius=1.5pt];
	    \node [below] at (1,0) {$x_1^2 - x_2^2 = 0$};
        \node[below] at (1,-1) {(a) split};

		\filldraw (4,1) circle[radius=1.5pt];
	    \node [below] at (4,0) {$x_1^2 + x_2^2 = 0$};
        \node[below] at (4,-1) {(b) non-split};
    \end{tikzpicture}
    \centering
\end{figure}

Here we may think of \textit{split} as corresponding to the existence of rational tangent directions, while \textit{non-split} refers to non-rational tangent directions. Over fields that aren't $\R$, it is possible to have many different split nodes.

In the case where $k = \C$, for any $(a_1, \ldots, a_n)$ sufficiently close to 0, it is a classical result that the number of nodes in this family is a constant integer, equal to $\deg_0^{\text{top}} \grad f =: \mu$, which is called the \textit{Milnor number}. This admits a generalization as follows.

\begin{theorem}\label{thm:A1-Milnor-number}\cite[Corollary~45]{KW-Milnor-Number} Assume that $f$ has a single isolated singularity at the origin. Then for a generic $(a_1,\ldots,a_n)$, we have that the sum over nodes on the hypersurface $f + a_1 x_1 + \ldots + a_n x_n = t$ is
\begin{align*}
	\sum_{\substack{\text{nodes } p \\ \text{in family}}} \type(p) = \deg_0^{\A^1} \grad f =: \mu^{\A^1}_0 f.
\end{align*}
We refer to this as the \textit{$\A^1$-Milnor number}. We remark that the classical Milnor number can be recovered by taking the rank of the $\A^1$-Milnor number.
\end{theorem}

\begin{example} Let $f(x,y) = x^3 - y^2$, over a field of characteristic not equal to 2 or 3. Let $p=(0,0)$ be a point on the hypersurface $\{f=0\}$. We can compute $\grad f = (3x^2, -2y)$, and then we have that
\begin{align*}
	\deg \grad f &= \deg(3x^2) \cdot \deg(-2y) \\
	&= \begin{pmatrix} 0 & 1/3 \\ 1/3 & 0 \\ \end{pmatrix} \gw{-2} \\
	&= \H.
\end{align*}

This has rank two, so the classical Milnor number is $\mu = 2$. We can take our family to be $y^2 = x^3 + ax + t$. If $a=0$, then we have a node at $0$. In general, for $a\neq 0$, we have nodes at those $t$ with the property that the discriminant of the curve $y^2 = x^3 + ax + t$ vanishes, that is at those $t$ where $\Delta = -16 \left( 4a^3 + 27t^2 \right) = 0$. This has at most two solutions in $t$, which we may denote by $\left\{ x^2 + u_1 y^2 = 0 \right\}$ and $\left\{ x^2 + u_2 y^2 = 0 \right\}$, and we see by Theorem~\ref{thm:A1-Milnor-number} that $\H = \left\langle u_1 \right\rangle + \left\langle u_2 \right\rangle$. This implies, by taking determinants, that $-1$ agrees with $u_1 u_2$ up to squares. This provides us with obstructions to the existence of pairs of nodes of certain types, depending on the choice of field we are working over. For example:

\begin{itemize}
    \item Over $\F_5$, we see that $\left\langle 1 \right\rangle = \left\langle -1 \right\rangle$ in $\GW(\F_5)$ implying that $u_1 u_2$ is always a square. In particular, $u_1$ and $u_2$ cannot have the property that exactly one of them is a non-square, meaning that we cannot bifurcate into a split and a non-split node.
    \item Over $\F_7$ we have that $\gw{1} \neq \gw{-1}$, implying $u_1 u_2$ is a non-square, so we cannot bifurcate into two split or two non-split nodes.
\end{itemize}
\end{example}

\begin{exercise} Compute $\mu^{\A^1}$ for the following ADE singularities over $\Q$:
\begin{center}
	\begin{tabular}{l | l}
	singularity & equation \\
	\hline
	$A_n$	&	$x^2 + y^{n+1}$ \\
	$D_n$	&	$y(x^2 + y^{n-2}) \quad (n\geq 4)$ \\
	$E_6$	&	$x^3 + y^4$ \\
	$E_7$	&	$x ( x^2 + y^3 )$ \\
	$E_8$	&	$x^3 + y^5$.
	\end{tabular}
\end{center}
\end{exercise}

\subsection{An arithmetic count of the lines on a smooth cubic surface}\label{sec:lines}

The following is based off of joint work of Jesse Kass and Kirsten Wickelgren \cite{KW-arithmetic-count}. Let $f \in k[x_0, x_1, x_2, x_3]$ be a homogeneous polynomial of degree three. Consider the following surface 
\[
	V = \{f = 0 \} \subseteq \Proj k[x_0, x_1, x_2 , x_3] = \P_k^3,
\]
and suppose that $V$ is smooth.

\begin{theorem}\label{thm:cayley-salmon} \textit{(Cayley-Salmon Theorem)} When $k=\C$, there are exactly 27 lines on $V$ \cite{cayley_2009}.
\end{theorem}

\begin{proof} Consider the Grassmannian $\Gr_\C(2,4)$, which parametrizes 2-dimensional complex subspaces $W\subseteq \C^{\oplus 4}$, or equivalently, lines in $\P_\C^3$. As the Grassmannian is a moduli space, it admits a \textit{tautological bundle} $\mathcal{S}$ whose fiber over any point $W\in \Gr_\C(2,4)$ is the vector space $W$ itself. A chosen homogeneous polynomial $f$ of degree three defines a section $\sigma_f$ of $\Sym^3 \mathcal{S}^\ast$, where
\begin{align*}
	\sigma_f \left( [W] \right) = \left. f \right|_W.
\end{align*}
Thus we see that the line $\ell \subseteq \P_\C^3$ corresponding to $[W]$ lies on the surface $V$ if and only if $\sigma_f [W] = 0$. One may see that $\sigma_f$ has isolated zeros \cite[Corollary~6.17]{3264}, and thus we may express the Euler class of the bundle as

\begin{equation}\label{eqn:euler-class}
\begin{aligned}
    e(\Sym^3 \mathcal{S}^\ast) = c_4 (\Sym^3 \mathcal{S}^\ast) = \sum_{\ell} \deg^\top_\ell \sigma_f,
\end{aligned}
\end{equation}

where this last sum is over the zeros of $\sigma_f$. We determine $\deg^\top_\ell \sigma_f$ by choosing local coordinates near $\ell$ on $\Gr_\C(2,4)$ as well as a compatible trivialization for $\Sym^3 \mathcal{S}^\ast$ over this coordinate patch. Then $\sigma_f$ may be viewed as a function
\begin{align*}
	\A_\C^4 \supseteq U \xto{\sigma_f} \A_\C^4
\end{align*}
with an isolated zero at $\ell$. We can then define $\deg^\top_\ell \sigma_f$ as the local degree of this function. It is a fact that the smoothness of $V$ implies that $\sigma_f$ vanishes to order 1 at $\ell$. Thus the Euler class counts the number of lines on $V$. Finally, one may compute $c_4(\Sym^3 \mathcal{S}^\ast) = 27$ by applying the splitting principle and computing the cohomology of $\Gr_\C(2,4)$.
\end{proof}

In the real case, Sch\"{a}fli \cite{Schafli} and Segre \cite{Segre} showed that there can be 3, 7, 15, or 27 real lines on $V$. One of the main differences between the real and the complex case was the distinction that Segre drew between \textit{hyperbolic} and \textit{elliptic} lines.

\begin{definition} We say that $I\in \PGL_2(\R)$ is \textit{hyperbolic} (resp. \textit{elliptic}) if the set
\begin{align*}
	\Fix(I) = \{x \in \P^1_\R \ : \ Ix = x\}
\end{align*}
consists of two real points (resp. a complex conjugate pair of points).
\end{definition}

To a real line $\ell \subseteq V$ we may associate an involution $I\in \Aut(\ell) \cong \PSL_2(\R)$, where $I$ sends $p\in \ell$ to $q\in \ell$ if $T_p V \cap V = \ell \cup Q$, for some $Q$ satisfying $\ell \cap Q = \{p,q\}$, (that is, for any point $p$ on a line $\ell$, there is exactly one other point $q$ having the same tangent space). We can say that $\ell$ is hyperbolic (resp. elliptic) whenever $I$ is.

Alternatively, we may describe these classes of lines topologically. We think of the frame bundle as a principal $\SO(3)$-bundle over $\RP^3$. As $\SO(3)$ admits a double cover $\Spin(3)$, from any principal $\SO(3)$-bundle we may obtain a principal $\Spin(3)$-bundle. Traveling on our cubic surface along the line $\ell$ gives a distinguished choice of frame at every point on $\ell$, that is, a loop in the frame bundle. This loop may or may not lift to the associated $\Spin(3)$-bundle. If the loop lifts, then $\ell$ is hyperbolic, and if it doesn't then $\ell$ is elliptic.

\begin{theorem}\label{thm:hyperbolic-minus-elliptic} In the real case, we have the following relationship between hyperbolic and elliptic lines:
\begin{align*}
	\#\{\text{real hyperbolic lines on } V \} - \#\{\text{real elliptic lines on } V \} = 3.
\end{align*}
We refer the reader to the following sources \cite{Segre,Benedetti-Silhol,Horev-Solomon,Okonek-Teleman,Finashin-Kharlamov}.
\end{theorem}
\begin{proof}[Proof sketch] Via the map $\sigma_f : \Gr_\R(2,4) \to \Sym^3 \mathcal{S}^\ast$, we have that
\begin{align*}
	e(\Sym^3 \mathcal{S}^\ast) = \sum_{ \substack{ \ell \in \Gr_\R (2,4) \\ \sigma_f (\ell) = 0 }} \deg^\top_\ell \sigma_f.
\end{align*}
One may also show that
\begin{align*}
	\deg^\top_\ell \sigma_f = \begin{cases} 1 & \text{if }\ell \text{ is hyperbolic} \\ -1 & \text{if }\ell \text{ is elliptic}, \\ \end{cases}
\end{align*}
and compute that $e(\Sym^3 \mathcal{S}^\ast) = 3$ using the Grassmannian of oriented planes.
\end{proof}

To define a notion of hyperbolic and elliptic which holds in more generality, we introduce the \textit{type} of a line. As before, we let $V\subseteq \P_k^3$ be a smooth cubic surface, and consider a closed point $\ell\in \Gr_k(2,4)$, with residue field $L = k(\ell)$. We can then view $\ell$ as a closed immersion
\begin{align*}
	\ell\cong \P_L^1 \clhook \P_k^3 \otimes_k L.
\end{align*}

Given such a line $\ell \subseteq V$, we again have an associated involution:
\begin{align*}
    I = \begin{pmatrix} a & b \\ c & d \end{pmatrix} \in \PGL_2(L).
\end{align*}

Since $I$ is an involution, its fixed points satisfy $\frac{az+d}{cz+d} = z$, from which we can see they are defined over the field $L \left( \sqrt{D} \right)$, where $D$ is the discriminant of the subscheme $\Fix(I) \subseteq \P^1_L$.

\begin{definition}\label{defn:type-line} The \textit{type} of a line $\ell$ is the element of $\GW(k(\ell))$ given by
\begin{align*}
	\type(\ell) := \gw{D} = \gw{ad-bc} = \gw{-1} \deg^{\A^1}(I).
\end{align*}
We say a line is \textit{hyperbolic} if $\type(\ell) = \left\langle 1 \right\rangle$, and \textit{elliptic} otherwise.
\end{definition}

\begin{theorem}\label{thm:Thm2-KW-arithm}\cite[Theorem~2]{KW-arithmetic-count} The number of lines on a smooth cubic surface is computed via the following weighted count
\begin{align*}
	\sum_{\ell \subseteq V} \Tr_{k(\ell)/k} (\type(\ell)) = 15\cdot \gw{1} + 12\cdot \gw{-1}.
\end{align*}
\end{theorem}

\begin{remark} We may apply the previous theorem to observe the following results:
\begin{enumerate}
\item If $k=\C$, then by taking the rank, we obtain the Cayley-Salmon Theorem (\ref{thm:cayley-salmon}), stating that the number of lines on a cubic surface is 27.

\item If $k=\R$, then $\Tr_{\C/\R} \gw{1} = \gw{1} \oplus \gw{-1}$. Taking the signature, we recover Theorem \ref{thm:hyperbolic-minus-elliptic}, stating that the number of hyperbolic lines minus the number of elliptic lines is 3.
\end{enumerate}
\end{remark}

As a particular application, if we are working over a finite field $k=\F_q$, then its square classes are $\F_q^\times \Big/ (\F_q^\times)^2 \cong \{1,u\}$. Thus the type of a line $\ell$ over $\F_{q^a}$ is either $\gw{1}$ or $\gw{u_a}$, which by Definition \ref{defn:type-line} we call hyperbolic or elliptic, respectively.

\begin{corollary}\label{cor-ellip-hyperbolic-lines} \cite[Theorem~1]{KW-arithmetic-count} For any natural number $a$, we have that the number of lines on $V$ satisfies
\begin{align*}
	&\#\{\text{elliptic lines with field of definition } \F_{q^{ 2a + 1}} \} \\
    &+ \#\{\text{hyperbolic lines with field of definition } \F_{q^{ 2a }}\} \equiv 0 \pmod{2}.
\end{align*}
In particular when all the lines in question are defined over a common field $k$, we have that the number of elliptic lines is even.
\end{corollary}

In order to prove Theorem~\ref{thm:Thm2-KW-arithm}, one considers $\sigma_f$ to be a section of the bundle $\Sym^3 \mathcal{S}^\ast \to \Gr_k(2,4)$, and computes a sum over its isolated zeros, weighted by their local index. Over the complex numbers, this is precisely Equation \ref{eqn:euler-class}, which recovers the Euler number of the bundle. In a more general context, however, we will want to obtain an element of $\GW(k)$. This requires us to use an enriched notion of an \textit{Euler class}, described below.

\begin{digression} In this exposition, given a vector bundle $E\to X$ with section $\sigma$, we use the Euler class $e(E,\sigma)$ valued in $\GW(k)$ of \cite[Section 4]{KW-arithmetic-count}. In the literature, there are a number of other Euler classes which coincide with this definition in various settings. One may define this Euler class via Chow-Witt groups \cite{Barge-Morel} or oriented Chow groups \cite{Fasel08} as in the work of M. Levine \cite{Levine-enumerative-geometry}. In his seminal book, Morel defines the Euler class of a bundle $E \to X$ as a cohomology class in twisted Milnor-Witt $K$-theory $H^n(X; \mathcal{K}^{\text{MW}}_n(\det E^\ast))$ \cite{morel}, and when $\det(E^\ast)$ is trivial, one may relate these Euler classes up to a unit multiple via the isomorphism
\begin{align*}
    H^n(X; \mathcal{K}^{\text{MW}}_n(\det E^\ast)) \cong \widetilde{\text{CH}}(X, \det E^\ast).
\end{align*}
For more details, see the work of Asok and Fasel \cite{Asok-Fasel}. Other versions of the Euler class in $\A^1$-homotopy theory occur in the work of D\'eglise, Jin and Khan \cite{DJK} and the work of Levine and Raksit \cite{Levine-Raksit}. Many of these notions are equated in work of Bachmann and Wickelgren \cite{BW3}.
\end{digression}

\begin{definition} Let $X$ be a smooth projective scheme of dimension $r$, and let $\mathcal{E} \to X$ be a rank $r$ bundle. We say that $\mathcal{E}$ is \textit{relatively oriented} if we are given an isomorphism
\begin{align*}
	\Hom( \det TX, \det \mathcal{E}) \cong \mathcal{L}^{\otimes 2},
\end{align*}
where $\mathcal{L}$ is a line bundle on $X$.
\end{definition}

Suppose that $\sigma$ is a section of a relatively oriented bundle $\mathcal{E}$ with isolated zeros, and define $Z= \{\sigma = 0\}$ to be its vanishing locus. For each $x\in Z$, we will define $\deg^{\A^1}_x \sigma$ as follows:
\begin{enumerate}
    \item Choose \textit{Nisnevich coordinates} (\cite[Definition~17]{KW-arithmetic-count}) near $x\in Z$, that is, pick an open neighborhood $U\subseteq X$ around $x$, and an \'etale morphism $\phi : U \to \A_k^r$ such that $k(\phi(x)) \cong k(x)$.
    \item Choose a \textit{compatible oriented} trivialization $\mathcal{E}\big|_U$, that is, a local trivialization
    \begin{align*}
    	\psi : \mathcal{E}\big|_U \to \O_U^{\oplus r},
    \end{align*}
    such that the associated section $\Hom( \det TX, \det \mathcal{E})(U)$ is a square of a section in $\mathcal{L}(U)$. Then we have that $\psi\circ \sigma \in \O_U^{\oplus r}$ and there exists a $g\in (\mathfrak{m}_x^N)^{\oplus r}$, with $N$ sufficiently large, so that
    \begin{align*}
    	\psi\circ \sigma + g \in \phi^\ast \O_{\A_k^r}.
    \end{align*}
    Define $f:= \psi\circ\sigma + g$, and then we have that $f: \phi(U) \to \A_k^r$ has an isolated zero at $\phi(x)$. Since our trivialization was compatibly oriented, this definition is independent of the choice of $g$.
    
    \item Finally, we define $\deg^{\A^1}_x \sigma := \deg^{\A^1}_{\phi(x)} f \in \GW(k)$.
\end{enumerate}

\begin{definition} For a relatively oriented bundle $\mathcal{E} \to X$, and a section $\sigma$ with isolated zeros, we define the \textit{Euler class} to be
\begin{align*}
    e(E,\sigma) := \sum_{x:\sigma(x) = 0} \deg^{\A^1}_x \sigma.
\end{align*}
\end{definition}

In order to conclude the proof of Theorem~\ref{thm:Thm2-KW-arithm}, we must identify $\deg^{\A^1}_\ell \sigma_f$ with $\type(\ell)$. Then we are able to compute $e(\Sym^3 \mathcal{S}^\ast)$ using a well-behaved choice of cubic surface, for instance the Fermat cubic. For more details, see \cite[\S5]{KW-arithmetic-count}.

\begin{remark} Following our definition of an Euler class for a relatively oriented bundle, we include the following closely related remarks.
\begin{enumerate}
    \item Interesting enumerative information is still available when relative orientability fails. For an example of this in the literature, we refer the reader to the paper of Larson and Vogt \cite{Larson-Vogt} which defines relatively oriented bundles relative to a divisor in order to compute an enriched count of bitangents to a smooth plane quartic \cite{Larson-Vogt}.
    \item Given a smooth projective scheme over a field, one may push forward the Euler class of its tangent bundle to obtain an Euler characteristic which is valued in $\GW(k)$. A particularly interesting consequence of this is an enriched version of the Riemann--Hurwitz formula, first established by M. Levine \cite[Theorem~12.7]{Levine-enumerative-geometry} and expanded upon by work of Bethea, Kass, and Wickelgren \cite{Bethea-Kass-Wickelgren}.

\end{enumerate}
\end{remark}

Forthcoming work of Pauli investigates the related question of lines on quintic threefold \cite{Pauli-Quintic_3_fold}. We also refer the reader to work of M. Levine, which includes an examination of Witt-valued characteristic classes, including an Euler class of $\Sym^{2n-d} \mathcal{S}^\ast$ on $\Gr_k(2,n+1)$ \cite{MLevine-Witt-char-classes}, and results of Bachmann and Wickelgren for symmetric bundles on arbitrary Grassmannians \cite[Corollary~6.2]{BW3}. Finally, for a further investigation of enriched intersection multiplicity, we refer the reader to recent work of McKean on enriching B\'ezout's Theorem \cite{McKean}.

\subsection{An arithmetic count of the lines meeting 4 lines in space}\label{sec:four-lines}
The following is based off of work by Padmavathi Srinivasan and Kirsten Wickelgren \cite{Padma-Kirsten}.

In enumerative geometry, one encounters the following classical question: given four complex lines in general position in $\CP^3$, how many other complex lines meet all four? The answer is two lines, whose proof we sketch out below.

\begin{customenvironment}{Four lines in three-space, classically} Let $L_1, L_2, L_3, L_4$ be lines in $\CP^3$ so that no three of them intersect at one point (we refer to this condition as \textit{general}). Given a point $p\in L_1$, there is a unique line $L_p$ through $p$ which intersects both $L_2$ and $L_3$. We then examine the surface sweeped out by all such lines $Q := \bigcup_{p\in L_1} L_p$, and we claim that this is a degree two hypersurface which contains $L_1$, $L_2$, and $L_3$. To see this, it suffices to verify that it is the vanishing locus of a degree two homogeneous polynomial. A homogeneous polynomial of degree two, considered as an element of $H^0(\CP^3, \O(2))$, will vanish on the line $L_i$ if and only if it lies in the kernel
\begin{align*}
    H^0(\CP^3, \O(2)) \to H^0(L_i, \O(2)).
\end{align*}
We verify that
\begin{align*}
	\dim_k H^0( \CP^3, \O(2)) = \binom{2+3}{2} = 10 \\
	\dim_k H^0(L_i, \O(2)) = 3,
\end{align*}
therefore for $i=1,2,3$ each such map has kernel of dimension $\ge 7$. This implies there is a polynomial $f$ in the common kernel of all three maps. We claim that $L_p \subseteq V(f)$ for each $p\in L_1$, and indeed since three points of $L_p$ lie in $V(f)$, we see that $V(f)$ contains the entire line. Therefore we have containment $V(f) \supseteq Q$, and it is easy to see we must have equality. Finally by applying B\'ezout's Theorem, we see that $Q\cap L_4$ consists of two points, counted with multiplicity.
\end{customenvironment}

One might ask how to answer this question over an arbitrary field $k$. We recall that the Grassmannian $\Gr_k(2,4)$ parametrizes lines in $\P_k^3$ (that is, two-dimensional subspaces of $k^{\oplus 4}$), which is an appealing moduli space for this problem. We first select a basis $\{e_1,e_2,e_3,e_4\}$ of $k^{\oplus 4}$ satisfying 
\begin{align*}
	L_1 &= k e_3 \oplus k e_4,
\end{align*}
and we define a new line $L$ such that
\begin{align*}
	L &= k \tilde{e}_3 \oplus k \tilde{e}_4,
\end{align*}
where $\tilde{e}_3$ and $\tilde{e}_4$ are some linearly independent vectors whose definition we defer until further below. Letting $\phi_i$ denote the dual basis element to $e_i$, one may compute that $L \cap L_1$ is nonempty if and only if 
\[
	(\phi_1 \wedge \phi_2)(\tilde{e}_3 \wedge \tilde{e}_4) = 0.
\]
Consider the line bundle $\det \mathcal{S}^\ast = \mathcal{S}^\ast \wedge \mathcal{S}^\ast \to \Gr_k(2,4)$, whose fiber over a point $W\in \Gr_k(2,4)$ is $W^\ast \wedge W^\ast$. We then have that $\phi_1 \wedge \phi_2 \in H^0( \Gr_k(2,4), \mathcal{S}^\ast \wedge \mathcal{S}^\ast)$ and
\begin{align*}
	(\phi_1 \wedge \phi_2)([W]) = \left. \phi_1 \right|_W \wedge \left. \phi_2 \right|_W.
\end{align*}
It is then clear that we obtain a bijection between lines intersecting $L_1$ and zeros of $\phi_1 \wedge \phi_2$:
\begin{align*}
	\{ L \ : \ L \cap L_1 \neq \emptyset\} = \{ [W] \ : (\phi_1 \wedge \phi_2)([W]) = 0\}.
\end{align*}
We may repeat this process for each line to form a section $\sigma$ of $\oplus_{i=1}^4 \mathcal{S}^\ast \wedge \mathcal{S}^\ast$. Then the zeros of $\sigma$ will correspond exactly to lines which meet all four of our chosen lines:
\begin{align*}
	\{ L \ : \ L \cap L_i \neq \emptyset,\ i=1,2,3,4 \} = \{ [W] \in \Gr_k(2,4) \ : \ \sigma([W]) = 0\}.
\end{align*}

In particular, if $\sigma$ is a section of a relatively oriented bundle, then we may calculate an enriched count of lines meeting four lines in space, given by the Euler class
\begin{align}\label{eqn:enriched-count-line-bundle}
	e\left( \oplus_{i=1}^4 \mathcal{S}^\ast \wedge \mathcal{S}^\ast , \sigma\right) = \sum_{L \ : \ L\cap L_i \neq 0} \ind_L \sigma.
\end{align}

Denote by $\mathcal{E} = \oplus_{i=1}^4 \mathcal{S}^\ast \wedge \mathcal{S}^\ast$ our rank four vector bundle over $X := \Gr_k(2,4)$. Since $X$ is a smooth projective scheme of dimension four, we have that $(\det TX)^\ast \cong \omega_X \cong \O(-2)^{\otimes 2}$, and $\det\mathcal{E} \cong \left( \otimes_{i=1}^2 \mathcal{S}^\ast \wedge \mathcal{S}^\ast\right)^{\otimes 2}$. Therefore $\Hom(\det TX, \det \mathcal{E}) \cong w_X \otimes \det \mathcal{E} \cong \mathcal{L}^{\otimes 2}$, so $\mathcal{E}$ is relatively oriented over $X$, and Equation \ref{eqn:enriched-count-line-bundle} is a valid expression. In order to compute a local index of the section $\sigma$ near a zero $L$, we must first parametrize Nisnevich local coordinates near $L$. Here we define a parametrized basis of $k^{\oplus 4}$ by
\begin{align*}
	\tilde{e}_1 &= e_1 \\
	\tilde{e}_2 &= e_2 \\
	\tilde{e}_3 &= x e_1 + y e_2 + e_3 \\
	\tilde{e}_4 &= x' e_1 + y' e_2 + e_4.
\end{align*}

We then obtain a morphism from affine space to an open cell around $L$:
\begin{align*}
    \A^4_k = \Spec k[x,y,x',y'] &\to U \subseteq \Gr_k(2,4) \\
    (x,y,x',y') &\mapsto \spn\{\tilde{e}_3, \tilde{e}_4\}.
\end{align*}

Over this cell, we obtain an oriented trivialization of the bundle $\det \mathcal{S}^\ast$, given by $\tilde{\phi}_3 \wedge \tilde{\phi}_4$, where $\tilde{\phi}_i$ denotes the dual basis element to $\tilde{e}_i$. Under these local coordinates, we may compute the local index $\ind_L \sigma$ as the local $\A^1$-degree at the origin of the induced map $\A^4_k \to \A^4_k$. Suppose that 
\begin{align*}
	L_1 = \{ \phi_1 = \phi_2 = 0 \} = k e_3 \oplus k e_4.
\end{align*}

Then we have that $\sigma([W]) = \left( \left. \phi_1 \wedge \phi_2 \right|_{[W]}, \ldots\right)$. We see then that
\begin{align*}
	\left. (\phi_1 \wedge \phi_2)\right|_{k \tilde{e}_3 \oplus k \tilde{e}_4} &= \left( x \tilde{\phi}_3 + y \tilde{\phi}_4 \right) \wedge \left( x' \tilde{\phi}_3 + y' \tilde{\phi}_4 \right) \\
	&=(x y' - x' y) \tilde{\phi}_3 \wedge \tilde{\phi}_4.
\end{align*}
Thus we may exhibit $\sigma$ as a function
\begin{align*}
	f = (f_1,f_2,f_3,f_4):\A_k^4 &\to \A_k^4,
\end{align*}

where $f_1(x,y,x',y') = xy' - x'y$. Then in the basis $(x,y,x',y')$ we have that the Jacobian of $\sigma$ has its first column as:
\begin{align*}
	\Jac(f) = \det \begin{pmatrix} y' & \cdots \\ -x' & \cdots \\ -y & \cdots \\ x & \cdots \end{pmatrix}.
\end{align*}

\begin{customenvironment}{Question} Is there a geometric interpretation of $\ind_L\sigma = \deg^{\A^1}_L f$?
\end{customenvironment}

The intersections $L\cap L_i$ for $i=1,\ldots,4$ determine four points on $L \cong \P_{k(L)}^1$. Let $\lambda_L$ denote the cross-ratio of these points in $k(L)^\ast$. Denote by $P_i$ the plane spanned by $L$ and $L_i$. We note that planes $P$ in $\P_k^3$ correspond to subspaces $V \subseteq k(P)^{\oplus 4}$ where $\dim(V) = 3$. If $P$ contains the line $L=[W]$ then it corresponds to $W\subseteq V\subseteq k(P)^{\oplus 4}$, which in turn corresponds to $k(P)$-points of $\Proj\left( k(L)^{\oplus 4} \Big/ W \right) \cong \P_{k(L)}^1$. Thus we might think of the planes $P_i$ for $i=1,\ldots,4$ as 4 points on $\P_{k(L)}^1$. Let $\mu_L$ denote the cross-ratio of these points.

\begin{theorem} \cite[Theorem~1]{Padma-Kirsten} Let $L_1, L_2, L_3, L_4$ be four general lines defined over $k$ in $\P_k^3$. Then
\begin{align*}
	\sum_{\{ L \ : \ L\cap L_i \neq \emptyset \ \forall i\}} \Tr_{k(L)/k} \gw{\lambda_L - \mu_L} = \gw{1} + \gw{-1}.
\end{align*}

As a generalization, let $\pi_1, \ldots, \pi_{2n-2}$ be codimension 2 planes in $\P_k^n$ for $n$ odd. Then
\begin{align*}
	\sum_{\{L \ : \ L\cap \pi_i \neq \emptyset \ \forall i\}} \Tr_{k(L)/k} \det \begin{pmatrix} \cdots & c_i b_1^i & \cdots \\
    \cdots & c_i b_2^i & \cdots \\ 
    \end{pmatrix} = \frac{1}{2n} \binom{2n-2}{n-1} \H,
\end{align*}
where $c_i$ are normalized coordinates for the line $\pi_i \cap L$ (defined in \cite[Definition~10]{Padma-Kirsten}), and $[b_1^i, b_2^i] = L \cap \pi_i \cong \P_{k(L)}^1$. This weighted count is expanded in forthcoming work of the author, which provides a generalized enriched count of $m$-planes meeting $mp$ codimension $m$ planes in $(m+p)$-space \cite{wronski-paper}.
\end{theorem}

\begin{corollary} \cite[Corollary~3]{Padma-Kirsten} Over $\F_q$, we cannot have a line $L$ over $\F_{q^2}$ with
\begin{align*}
	\lambda_L - \mu_L = \begin{cases} \text{non-square} & q \equiv 3 \pmod{4} \\ \text{square} & q\equiv 1 \pmod{4}. \\ \end{cases}
\end{align*}
\end{corollary}

For related results in the literature, we refer the reader to the papers of Levine and Bachmann--Wickelgren mentioned in the previous section \cite{MLevine-Witt-char-classes,BW3}, as well as Wendt's work developing a Schubert calculus valued in Chow-Witt groups \cite{Wendt-Schubert-calc}.  Finally, Pauli uses Macaulay2 to compute enriched counts over a finite field of prime order and the rationals for various problems presented in these conference proceedings, including lines on a cubic surface, lines meeting four general lines in space, the EKL class, and various $\A^1$-Milnor numbers \cite{Sabrina-M2}.

\setcounter{secnumdepth}{0}
\section{Notation Guide}
\begin{longtable}{p{1.5cm}p{12cm}}
$\gw{a}$ & the element of the Grothendieck--Witt ring corresponding to $a \in k^\times / (k^\times)^2$ \\
$[X,Y]_{\A^1}$ & genuine $\A^1$-homotopy classes of morphisms $X\to Y$ \\
$\Bez(f/g)$ & B\'ezout bilinear form of a rational function \\
$\Bl_Z X$ & blowup of a subscheme $Z$ in $X$ \\
$\deg^{\A^1}$ & global $\A^1$-degree \\
$\deg^\top$ & the topological (Brouwer) degree of a map between real or complex manifolds \\
$e(\mathcal{F})$ & the Euler number (Euler class) of a vector bundle \\
$\G_m$ & the multiplicative group scheme \\
$\Gr_k(n,m)$ & the Grassmannian of affine $n$-planes in $m$-space over a field $k$ \\
$\GW(k)$ & the Grothendieck--Witt ring over a field $k$ \\
$\H$ & the hyperbolic element $\gw{1} + \gw{-1}$ in $\GW(k)$ \\
$\mu^{\A^1}$ & $\A^1$-Milnor number \\
$[X,Y]_N$ & naive $\A^1$-homotopy classes of morphisms $X\to Y$ \\
$N_Z X$ & the normal bundle of a subscheme $Z$ in $X$ \\
$\Sh_\tau (\mathscr{C})$ & the category of sheaves in a Grothendieck topology $\tau$ on a category $\mathscr{C}$ \\
$\s\Pre(\mathscr{C})$ & the category of simplicial presheaves on $\mathscr{C}$ \\
$\Th(V)$ & Thom space of a vector bundle $V$ \\
$\Tr_{L/K}$ & trace for a field extension $L/K$ \\
$w^\EKL$ & the Eisenbud-Levine/Khimshiashvili bilinear form \\
\end{longtable}

\printbibliography

\end{document}